\documentclass{article}

\usepackage{amssymb,amsbsy,amsmath,amsfonts,amssymb,amscd,amsthm}
\usepackage{xcolor}
\usepackage{bbm}
\DeclareMathOperator{\Tr}{Tr}

\newcommand{\ov}{\overline}
\newcommand{\1}{\mathbbm{1}}
\newcommand{\R}{\mathbb{R}}

\usepackage{graphicx}
\graphicspath{{./}{figures/}}

\newcommand{\beq}{\begin{equation}}
\newcommand{\eeq}{\end{equation}}
\def\a{\alpha}
\def\b{\beta}
\def\d{\delta}

\def\g{\gamma}
\def\l{\lambda}

\def\m{\mu}
\def\n{\nu}
\def\r{\rho}
\def\s{\sigma}

\def\th{\theta}

\def\t{\tau}
\def\e{\varepsilon}

\def\pd{\partial}
\def\half{\frac{1}{2}}

\newcommand{\cD}{{\cal D}}

\newcommand{\cN}{{\cal N}}

\newcommand{\diver}{{\rm div}}
\newcommand{\io}{\int_{\R^d}}

\newtheorem{theorem}{Theorem}[section]

\newtheorem{definition}[theorem]{Definition}
\newtheorem{proposition}[theorem]{Proposition}

\newtheorem{remark}[theorem]{Remark}

\numberwithin{equation}{section}

\titlepage
\title{A Mean Field Games approach\\to Cluster Analysis}
\date{}

\author{Laura Aquilanti, Simone Cacace, Fabio Camilli\\ and Raul De Maio}

\begin{document}
\maketitle
\begin{abstract}
In this paper, we develop a Mean Field Games approach to Cluster Analysis.  We consider a finite mixture model, given by a  convex combination of probability density functions,   to describe  the given data set.  We interpret
a data point as an agent  of one of the populations represented by the components of the mixture model, and we introduce a corresponding optimal control problem. In this way, we obtain a multi-population Mean Field Games system which characterizes the parameters  of the finite mixture model.  Our method can be interpreted as a continuous version of the classical Expectation-Maximization algorithm.
\end{abstract}

\noindent
{\footnotesize \textbf{AMS-Subject Classification:} 62H30, 35J47, 49N70, 91C20}.\\
{\footnotesize \textbf{Keywords:}  mixture model;  Cluster Analysis; Expectation-Maximization algorithm; Mean Field Games; multi-population model}.

\section{Introduction}
Cluster Analysis, a classical problem in unsupervised Machine Learning, concerns the repartition  of a   group of data  points into subgroups, in such a way that the elements within a same group are more similar to each other than they are to the elements of a different one. Clustering methods can be used either to infer  specific patterns in rough data sets,  or as a preliminary   analysis in supervised Machine Learning algorithms (for a review of  the applications of cluster  analysis, see \cite{Bishop,sax}).
Algorithms for Cluster Analysis are generically divided into two classes: hard clustering methods, whose prototype is the K-means algorithm and its variants, for which each data point belongs only to a single subgroup \cite{bb}; soft clustering methods, such as the  Expectation-Maximization and  the Fuzzy K-means  algorithms, for which each data point has a certain  degree or probability to belong to each cluster \cite{bilmes,Miyamoto}.\par
Most of the mathematical literature on  Cluster Analysis, and more in general on Machine Learning, is in the framework of finite-dimensional   optimization. But recently,   different approaches based on partial differential equations   and infinite dimensional optimization   have begun to be pursued \cite{chaudari,chen,e}. Even if, in most of the cases, it is not clear if these new approaches are computationally competitive with the   traditional ones, they offer new points of view  and insights in this rapidly developing field. \par
In this paper, we propose an approach to Cluster Analysis based on   Mean Field Games (MFG in short). This theory was introduced by Lasry and Lions \cite{Lions3} in the framework of differential games with a large number of rational agents. Although it was realized from the beginning  that, in perspective,  one of the most promising application of MFG theory could be  the  ``Big Data" analysis \cite{Lions_Paris_Princeton
}, up to now most of the applications  are in the fields of Economics and Financial Mathematics, and   only very few papers deal with data analysis problems  \cite{pequito, vassilaras}. Nevertheless, a  MFG version of the classical hard-clustering K-means algorithm has been recently    proposed in  \cite{coron}. Relying on a similar approach, here we deal with  a MFG version of the soft clustering Expectation-Maximization algorithm.  \par
As in the classical Expectation-Maximization (EM in short) algorithm, the starting point of our analysis is  a finite mixture model given by convex combinations of probability density functions  (PDFs in short) 
$$m(x)=\sum_{k=1}^K \a_k m_k(x), \ x \in \mathbb{R}^d, \quad \text{with}\quad\sum_{k=1}^K \a_k=1.$$ 
Finite mixture models are powerful probabilistic techniques  to describe complex scenarios, e.g. several regions with high mass density,  and   populations composed by different sub-populations, each  one represented by a PDF. Generally, the components of the mixture are assumed to be parametrized PDFs. For example, for a  Gaussian mixture model, the parameters are represented by the mean and the  variance  of each component. Then, the EM algorithm is a technique which allows to compute the unknown coefficients (weights and parameters) of the mixture by incrementally increasing the maximum likelihood   with respect to the data set.\par 
In our approach, we are given a PDF $f$ on $\R^d$ representing the data set, and we aim to find  a corresponding mixture model in order to subdivide the points in $K$ clusters. We interpret the data points as agents  belonging  with a certain probability   to one of the   $K$  sub-populations, which are described by a PDF $m_k$ with   appropriate weights $\a_k$ on the entire population. 
The similarity, or proximity,  among the members of a same sub-population  is encoded in the  cost functions of the optimal control problems for each sub-population, which 
 push the agents to aggregate around the closer  barycentre of the  distributions $m_k$.  \par
In order to characterize the components of the mixture model, i.e. the PDFs $m_k$ and the weights $\a_k$, we first consider the  case of a quadratic MFG system
\begin{equation}\label{MFG_intro}
\left\{
\begin{array}{ll}
-\e  \Delta u_k(x)+ \half |Du_k(x)|^2+\l_k=\half (x-\mu_k)^t(\Sigma_k^{-1})^t\Sigma_k^{-1}(x-\mu_k)  ,\,& x\in\R^d, \\[8pt]
\e  \Delta m_k(x)+\diver (m_k(x)Du_k(x) )=0,&x\in\R^d, \\[8pt]
\a_k=\int_{\R^d }\g_k(x) f(x)dx, \ \\[8pt] 
m_k\ge 0,\,\io m_k(x)dx=1,    u_k( \m_k)=0,
\end{array}
\right.
\end{equation}
for $k=1, \dots, K$, where the coupling among the different populations is encoded in the right hand side of the first equation. Indeed, the mean $\mu_k$ and the covariance matrix  $\Sigma_k$ of the population $m_k$ depend, besides  data set $f$,  on the whole measure $m$ through the quantities $\gamma_k(x)=\a_k m_k(x)/m(x)$. The latter play a central role in the model because they represent the degree of membership of agents to a given population $k$.
We introduce a notion of consistency with the data set, and we show that a solution of the multi-population MFG system \eqref{MFG_intro} is given by a consistent Gaussian mixture with the same mean and variance of $f$. Moreover, any consistent Gaussian mixture can be obtained as a solution of the MFG system.
We remark that the main difference between our approach and the classical   EM algorithm, which is an iterative procedure, is that the MFG method    gives the partition into clusters as the result of the solution of a stationary system of PDEs.\par
Then we consider a     stationary multi-population  MFG system,  where the Hamiltonian and the coupling terms can be very general, and we show that it always admits a solution. Note that, in this framework, it is not reasonable  to obtain uniqueness, since   there may be several admissible partitions of a given data set.
We also consider a two-step numerical approximation of the MFG system, inspired to the classical EM algorithm, and we test the method on discrete and continuous data sets.\par
The paper is organized as follows. In Section \ref{sec:Algorithm}, following \cite {Bishop,bilmes}, we briefly introduce the  K-means and the Expectation Maximization algorithms, in their classical approach. In Section \ref{sec:quadraticMFG}, we first describe the approach in \cite{coron} and then we propose  a quadratic  multi-population MFG model for   Gaussian finite mixture models. In Section \ref{sec:general_problem}, we  study a  general class of MFG systems for Cluster Analysis. Finally, in Section \ref{sec:Numericalanalysis}, we perform some numerical tests to confirm the theoretical results.
%

\section{Clustering Algorithms}\label{sec:Algorithm}
In this section, we briefly review some classical algorithms for  Cluster Analysis. We start by describing the    {\it K-means algorithm},    a hard clustering technique,  hence     the {\it Expectation-Maximization algorithm}, a soft clustering one. For a complete introduction to the subject, we refer to    \cite[Cap. 9]{Bishop}.\\
In   the following, we are given a set   of observable data $\mathcal{X}=\{x_1, \dots, x_I\}$ in $\R^d$ and a fixed number $K$ of clusters $S_j, \ j=1, \dots, K$, into which we want to partition our observations.\par
\subsection{K-Means algorithm}
The K-means algorithm is an iterative procedure that assigns points to the nearest cluster with respect to the Euclidean  distance, even though other distances can be considered. We denote by   $\mu= \{\mu_1, \mu_2, \dots, \mu_K\}$, $\mu_k \in \R^d$,    the vector of \textit{barycentres} of the clusters $\{S_1,S_2,...,S_K\}$. Moreover, we introduce  a vector of cluster assignment, $c =(c_{1}, c_{2}, \dots, c_{I})$,  $c_i \in \{1, \dots K \}$, such that for every vector $x_i \in \mathcal{X}$, we have   $c_i= k \Leftrightarrow x_i \in S_k.$\\
The goal is to find the vectors  $ \mu$ and $c$   in order to minimize the functional 
\begin{equation}\label{i_funckmean}
J=\sum_{i=1}^{I} \sum _{k=1}^{K}{\1_{\{c_{i}=k\}}}\vert  x_i-\mu_k \vert^2\,,
\end{equation}
representing the sum of the distances of the data points to the corresponding barycentres ($\1$ denotes the characteristic function).
Starting from an arbitrary assignment for the vector $ \mu^0$,   the \mbox{$n^{th}$} iteration of the K-means algorithm alternates two steps:
\begin{itemize}
	\item[]\textbf{Cluster assignment}: 
	Minimization of the objective function $J$ with respect to $c$ by assigning the point $x_i$ to the closest barycentre, i.e.
	\begin{equation*}
	c^{n}_i= {\arg\min}_j \vert  x_i -\mu^{n}_j\vert^2 \qquad \forall i=1, \dots, I.
	\end{equation*}
\item[]\textbf{Barycentre update}:
	Computation of the new barycentres   $ \mu^{n+1}_k$ of the cluster $S^n_k=\{x_i\in\mathcal{X}:\,c^n_i=k\}$, i.e.   
	\begin{equation*}
	\mu^{n+1}_k= \dfrac{\sum_{i=1}^{I} x_i\1_{\{c^n_{i}=k\}}}{\sum_{i=1}^{I}\1_{\{c^n_{i}=k\}}}= \dfrac{\sum_{i=1}^{I} x_i \1_{\{c^n_{i}=k\}}}{\text{card}(S^n_k)} \qquad \forall k=1, \dots, K.
	\end{equation*}
\end{itemize} 
At each iteration, the objective function $J$ decreases and the algorithm  is repeated until there is no change in the cluster assignments, up to a tolerance error. \par
\subsection{Finite mixture model and EM algorithm}
Mixture models are probabilistic techniques used to represent   complex scenarios of data, and they are
related to soft clustering methods which consider a probabilistic partition of the data points into clusters, whose corresponding parametrized probabilities  are  computed via the EM   algorithm.
In this setting, we assume that  the data set $\mathcal{X}$ represents a set of (independent and identically distributed) observations of a random variable $X$, with unknown distribution that we aim to approximate with a distribution $P$. Moreover, we assume that $P$ is   given by a mixture of probability densities $p_k(x|\theta_k)$,  depending on some parameters $\theta_k$,  with  mixing coefficients $\a_k$ such that $\sum_{k=1}^K\a_k=1$, i.e. 
\begin{equation}\label{i_mixture}
P(x)= \sum_{k=1}^{K} \a_k p_k(x | \theta_k).
\end{equation}
The  purpose is to compute the parameters  $\a=(\a_1,\dots,\a_k)$, $\th=(\th_1,\dots,\th_k)$ of the mixture model \eqref{i_mixture} in order to maximize the likelihood objective function 
\begin{equation*}
\ln P(\mathcal{X}| \a, \theta)= \sum_{i=1}^{I}\left(\ln \sum_{k=1}^{K}\a_k p_k(x_i| \theta_k)\right) 
\end{equation*} 
with respect to data set $\mathcal{X}$. 
To simplify the computation of the extrema of the previous functional,  the data set $\mathcal{X}$ is considered as incomplete, and it is introduced a random variable $\mathcal{Y}= \{y_i\}_{i=1}^{I}$ whose values $y_i$ specify the component of $P$  generating  the  point, i.e. $y_i\in \{1, \dots, K\}$ and $y_i= k$ if and only  if the  point $x_i$ has  been generated by the distribution $p_k$. We define the   responsibility of $x_i$ with respect to the k-th cluster as
	\begin{equation}\label{i_resp}
	\gamma_{k}(x_i):=p_k(y_i=k| x_i, \theta_k),
	\end{equation}
	i.e. as the probability of an assignment $y_i$ once we have observed $x_i$;  the EM algorithm aims to find a maximum of the expected value of the complete data log-likelihood functional
	\begin{equation}\label{i_ML}
	\mathbb{E}_{\mathcal{Y}}[\ln p(\mathcal{X, Y}|  \theta_k, \alpha_k )]= \sum_{i=1}^{I}\sum_{k=1}^{K}\g_{ k}(x_i)\ln (\a_k p_k(x_i| \theta_k) )
	\end{equation}
	(see formula (9.40) in \cite{Bishop}). 
It takes a particular simple and explicit form if we consider  parametric distributions $p_k(x| \theta_k)$ given  by Gaussian distributions $\cN(x|\m_k,\Sigma_k)$. Starting from an arbitrary initialization    $\mu^0_k$, $\Sigma^0_k$, $\a_k^0$, $k=1, \dots, K$, and  computing the optimality conditions for the functional \eqref{i_ML} with respect to $\g_k$ for $\a_k,\m_k,\Sigma_k$ fixed,  and then exchanging the roles,  we get the following two alternating steps: 
\begin{itemize}
	\item[]\textbf{E-step}:	 Given $\mu^{n-1}_k$ $\Sigma^{n-1}_k$, $\a^{n-1}_k$, $k=1, \dots, K$,  compute the posterior probability (responsibility) of a datum $x_i$ to belong to  the  cluster $S_k$
	\begin{equation*}\label{i_EM_resp}
	\g^n_{k}(x_i)= P(y_i=k| x_i, \mu^{n-1}_k, \Sigma^{n-1}_k)=\dfrac{\a^{n-1}_k \mathcal{N}(x_i| \mu^{n-1}_k, \Sigma^{n-1}_k)}{\sum_{j=1}^{K}\a^{n-1}_j \mathcal{N}(x_i| \mu^{n-1}_j, \Sigma^{n-1}_j)}.
	\end{equation*}
	
	\item[]\textbf{M-step}:  
	Update the parameters $\a$, $\m$, $\Sigma$, by setting 	for  $k=1, \dots, K$,
	\begin{align}
	&\a^{n}_k= \dfrac{\sum_{i=1}^I \g^n_{k}(x_i)}{I}, \nonumber\\[3pt]
	 &\mu^{n}_k= \dfrac{\sum_{i=1}^{N} x_i\g^n_{k}(x_i)}{  \sum_{i=1}^I \g^n_{k}(x_i)},\label{i_em_mean}\\[3pt]
	 &\Sigma^{n}_k= \dfrac{\sum_{i=1}^{N}\g^n_{k}(x_i)(x_i-\mu^{n}_k)(x_i-\mu^{n}_k)^t}{ \sum_{i=1}^I \g^n_{k}(x_i)}. \label{i_em_var}
	\end{align}

\end{itemize}
After the convergence of the algorithm under some appropriate stopping criterion, we interpret the responsibility $\g_k(x_i)$ as the probability that $x_i$ belongs to the cluster $S_k$.\par
It is  interesting to observe that the K-means algorithm can be derived as a limit of the EM algorithm when the variances of the mixture model \eqref{i_ML}   go to zero. 
Indeed, assume for simplicity that the covariance matrix is  $\Sigma_k= \sigma I$, $ k=1, \dots, K$ , where $I$ is the identity $d\times d$ matrix  and  $\sigma >0$ a constant. In this case, the  responsibilities in \eqref{i_resp} are given by
	\begin{equation*}
	\g_{ k}(x_i)= \dfrac{\a_k e^{-\frac{\vert  x_i-\mu_k  \vert^2}{2 \sigma }}}{\sum_{j=1}^{K}\a_j e^{-\frac{ \vert x_i-\mu_j \vert^2}{2 \sigma}}}.
	\end{equation*}
 For  $\sigma \rightarrow 0$, the responsibility  $\g_{k}(x_i)$, corresponding to index $k$   minimizing the distance of the barycentre $\m_k$ from $x_i$, tends to $1$, while,  for $j\neq k$, all the other responsibilities $\g_{j}(x_i)$  go to $0$. 
Hence, at least formally,     $\g_{k}(x_i)\rightarrow \1_{\{c_i=k\}}$ and the maximum likelihood functional \eqref{i_ML} tends, up to a constant, to the K-means functional $J$ defined in \eqref{i_funckmean} (see  \cite[Sec. 9.3.2]{Bishop}).

\section{MFG models for Cluster Analysis}\label{sec:quadraticMFG}
In this section, we describe some models for Cluster Analysis based on MFG theory. In the classical MFG theory (see \cite{Lions3}), we have a continuum of indistinguishable agents whose  dynamics,    driven by independent Brownian motions $W_t$, is given by
\begin{equation}\label{mfg_dynamics}
\begin{cases}
d X_t= a_tdt + \sqrt{2\epsilon} d W_t,\qquad t>0 \\
X_0= x.
\end{cases}
\end{equation}
The function $a_t$ represents  the control which an agent chooses in order to minimize a cost function $J$, to be specified according to the considered model. Most of the time, the cost $J$ comprises a term $L(x,a)$    depending on the current state $x$ and control $a$ of the agent,  and a term $F(x,m)$ depending on the distribution $m$ of the other agents. For example,     consider the long time average cost
\begin{equation*}\label{mfg_ergodic}
J(x,a)= \lim_{T\to +\infty} \dfrac{1}{T} \mathbb{E}_x \bigg\{\int_{0}^{T}\big[L(X_s,a_s)+F(X_s,m(X_s))\big]ds \bigg\},
\end{equation*}
where $m$ is a density function representing the distribution of the agents.
Let $H(x,p)=\sup_{q\in\R^d}\{pq-L(x,q)\}$ be the Hamiltonian function. Then the triplet $(u,m,\l)$, given by the ergodic cost  $\l$, the corrector $u$ and the   distribution of the agents $m$ corresponding to the choice of the optimal control $D_pH(x,Du)$, can be characterized as  a solution of the   MFG system
\begin{equation}\label{mfg_MFG}
\left\{
\begin{array}{ll}
-\e \Delta u (x)+ H (x,Du (x))+\l=F (x,m ),\quad& x\in\R^d, \\[6pt]
\e  \Delta m (x)+\diver (D_pH(x,Du(x))m (x)  )=0,&x\in\R^d, \\[6pt]
m \ge 0,\ \io m (x)dx=1,\ \int_{\mathbb{R}^d}u (x)dx=0. 
\end{array}
\right.
\end{equation}
The first equation is a Hamilton-Jacobi-Bellman equation for the couple $(\l,u)$; the second is a Fokker-Planck equation for the distribution $m$,  which provides the density   of the population  at the position $ x \in \mathbb{R}^d$;   the normalization
conditions are imposed    since   $u$ is defined up to a constant, and   $m$ must represent a probability density function. \par
In some problems, as the ones we consider, several populations of agents, homogeneous within a given group but with different objectives and preferences from one group to the other,  can interact in the same environment. In this case, the model can be described by a   multi-population MFG system, where, for  $k=1, \dots, K$, each population  satisfies
a system such as \eqref{mfg_MFG}, but with 
 the cost function  $F_k$   depending on $m_k$ and  also on the distributions $m_j$, $j\neq k$ of the other populations (see \cite{cirant,Lions3}).\par
We   describe two approaches to Cluster Analysis based on  MFG theory, which correspond respectively to the K-means algorithm and to the EM algorithm. We consider  a data set   given by the support of a measure with density function $f: \R^d\rightarrow \R$ such that
\begin{equation}\label{mfg_density}
   f\ge 0  ,\quad  \int_{\R^d} f(x)d x =1.
\end{equation}
Note that the data set $\mathcal{X}$ in Section \ref{sec:Algorithm} corresponds to the case  of an atomic measure
\begin{equation}\label{atomic_meas}
P(x)= \frac{1}{I}\sum_{i=1}^{I}\delta_{x_i}(x).
\end{equation}
 The  function $f$ describes the entire population of agents/data points, that we want to subdivide in    $K$   sub-populations/clusters  on the basis of some similar properties or characteristics, where $K$ is fixed a priori.
Each   population is represented by a density function $m_k$ and the similarity or proximity will be expressed by means of an appropriate cost function. 
\subsection{The MFG K-means model}\label{subsec:coron}
A MFG approach   which mimics the K-means algorithm has been recently introduced in \cite{coron}.
Since this approach is the basis for the EM algorithm we will subsequently study, we  briefly describe it. It is worth noting that it is a hard clustering method, as   the classical K means algorithm, hence each data point can belong only to a single cluster $S_k$, $k=1,\dots,K$. \\
Denote by $m=(m_1,\dots,m_K)$ and  $u=(u_1,\dots,u_K)$ the vectors of the density functions and of the corresponding value functions. An agent of the $k$-th population moves according to the dynamics \eqref{mfg_dynamics}  and minimizes  the  infinite horizon cost functional
	\begin{equation*}
	u_k(x)=  \inf_{a_s}\,  \mathbb{E}_x \bigg[ \int_{0}^{\infty}\bigg(\half|a_s|^2+F_k(X_s,m_k(X_s))\bigg)e^{-\r s}ds\bigg] ,
	\end{equation*}
	where $\r>0$, the discount factor, is a fixed constant. The coupling cost is given by
	\begin{equation}\label{kmfg_cost}
	F_k(x,m_k)=\frac{1+\r}{2}|x-y_k|^2,
	\end{equation}
	where the barycentres $y_k$ are defined by
	\begin{equation}\label{kmfg_bary}
	y_k:=\dfrac{\int_{\R^d} xm_k(x)dx}{\int_{\R^d} m_k(x)dx}.
	\end{equation} 
Consider the multi-population MFG system for $k=1,\dots,K$,
		\begin{equation}\label{kmfg_MFG}
	\left\{
	\begin{array}{ll}
	\r u_k-\e \Delta u_k (x)+ \half |Du_k(x)|^2=F_k (x,m_k),\quad& x\in\R^d, \\[6pt]
	\r m_k(x)-\e  \Delta m (x)-\diver (Du_k(x) m_k (x)  )=\rho \tilde{f}_k &x\in\R^d.
	\end{array}
	\right.
	\end{equation}
where $\tilde{f}_k$ is a Gaussian distribution with mean
\begin{equation}\label{tilde_barycentres}
\tilde{y}_k= \frac{\int_{S_k}x f(x)dx}{\int_{S_k} f(x) dx}
\end{equation} and variance $\e$, and the cluster $S_k=S_k(u)$ related to $\tilde{y}_k$ is defined by
\begin{equation}\label{kmfg_cluster}
S_k=\{x\in\R^d:u_k(x)=\min_{j=1,\dots,K}u_j(x)\}.
\end{equation}
The value function $u_k$ is intended as a measure of the distance from the barycentre $y_k$.	We refer to \cite{coron} for a detailed interpretation of the system. \\
 A family of vectors $\{\bar{y}_k\}_{k=1}^K$   is said to satisfy the self consistency rule
 with the data set $f$ if
	\begin{equation}\label{kmfg_regle}
	\bar{y}_k=\dfrac{\int_{S_k}xf(x)dx}{\int_{S_k}f(x)dx}
	\end{equation}
for a family of clusters $\{S_k\}_{k=1}^K$.\\  
	In \cite{coron}, the following two crucial results are proved:
	\begin{itemize}
		\item[(i)] For all the family of vectors   $\{\bar y_k\}_{k=1}^K$  satisfying the self consistency rule \eqref{kmfg_regle}  for some family of clusters $\{S_k\}$, there exists a solution $(u,m)$ of \eqref{kmfg_MFG} such that $\bar y_k=y_k$ with $y_k$ given by \eqref{kmfg_bary}. 
		\item[(ii)] Given a solution  $(u,m)$ of \eqref{kmfg_MFG}, then the barycentres $y_k$ given by  \eqref{kmfg_bary} satisfy the self consistency rule
		\eqref{kmfg_regle} for $S_k=S_k(u)$ defined in \eqref{kmfg_cluster}.
	\end{itemize}
	It is important  to observe that it is not reasonable to expect uniqueness of the solution to \eqref{kmfg_MFG}, since  there exist  in general several sets of barycentres $\{\bar y_k\}_{k=1}^K$ satisfying \eqref{kmfg_regle}. This non-uniqueness property can be interpreted as the convergence to a local minimum in the classical K-means algorithm. 

\subsection{The Gaussian MFG EM model}
The aim is to find,  by means of an appropriate multi-population MFG system, a mixture of $K$ density functions    which gives an  optimal representation of a data set, described by a density function $f$ as in \eqref{mfg_density}. The resulting method is soft-clustering, since each point/agent has a certain probability to belong to each one of the clusters. \\
In this section, in order to explain the method in a simple setting,  we will describe a specific model for a Gaussian mixture, while the general method will be described in Section \ref{sec:general_problem}.
\\
We introduce some preliminary notations. Given a mixture \begin{equation}\label{emmfg_mix}
	m(x)=\sum_{k=1}^K \a_k m_k(x)
\end{equation}
where $\a_k\in (0,1)$, $\sum_{k=1}^K\a_k=1$ and
$m_k\ge0$, $\int_{\R^d}m_kdx=1$,
we introduce the   responsibilities
\begin{equation}\label{emmfg_resp}
	\g_k(x)=\dfrac{\a_k m_k(x)}{m(x)},\qquad x\in\R^d,\, k=1,\dots,K.
\end{equation}
We also define 
\begin{align}
	&\mu_k=\frac{\int_{\R^d}x\gamma_{k}(x)f(x)dx}{\int_{\R^d} \gamma_{k}(x)f(x)dx},\label{emmfg_mean}\\[12pt]
	&\Sigma_k=\frac{\int_{\R^d}(x-\mu_k)(x-\mu_k)^t\gamma_{k}(x)f(x)dx}{\int_{\R^d} \gamma_{k}(x)f(x)dx}\label{emmfg_var}
\end{align} which are the mean $\mu_k \in \mathbb{R}^d$ and the covariance matrix $\Sigma_k \in \mathbb{R}^{d \times d}$ of the density function $m_k$ computed with respect to the data set $f$.
Since, as we will see in the rest of the section, the function $m$ is given by   a convex combination of Gaussian functions with positive weights $\alpha_k$, then $\gamma_k>0$ in $\mathbb{R}^d$ and therefore also $\mu_k$ and $\Sigma_k$ are well defined. 
\begin{remark}
	Consider a finite data set $\mathcal{X}=\{x_1, \dots, x_I\}$  and the corresponding atomic measure $P$ defined in \eqref{atomic_meas}. Then, replacing the measure $f(x)dx$  in \eqref{emmfg_mean}-\eqref{emmfg_var} with $P$, we obtain
	\begin{align*}
		&\m_k=\frac{\sum_{i=1}^I \gamma_k(x_i)x_i}{\sum_{i=1}^I \gamma_k(x_i)}, \quad
		\Sigma_k=\frac{\sum_{i=1}^I \gamma_k(x_i)(x_i-\mu_i)(x_i-\mu_i)^t}{\sum_{i=1}^I \gamma_k(x_i)},
	\end{align*}
	which coincide with formulas \eqref{i_em_mean},\eqref{i_em_var} used in the Maximization step of the EM algorithm.
\end{remark}
Given  the same type of dynamics \eqref{mfg_dynamics}  for all the populations, an agent of the $k$-th population  wants to minimize the ergodic cost functional 
\begin{equation}\label{emmfg_ergodic}
	J_k(x,a)= \lim_{T\to +\infty}   \dfrac{1}{T}\mathbb{E}_x\left\{\int_{0}^{T}\bigg[\half|a_t|^2+F_k(X_s,m_k(X_s),m(X_s))\bigg]ds\right\},
\end{equation}
where
\begin{equation}\label{emmfg_cost}
	F(x,m_k, m)=\half (x-\mu_k)^t(\Sigma_k^{-1})^t\Sigma_k^{-1}(x-\mu_k)
\end{equation}
Note that the coupling among the various populations is given by the dependence of $\m_k$, $\Sigma_k$ on the responsibility $\g_k$, which in turn depends on the measure $m$.  Observe that $F_k$ is instead independent of $u_j$, $j=1,\dots,K$.
The coupling term   forces a generic   data point  to distribute with an higher probability    around the nearest mean point $\m_k$, since it penalizes the square of the distance from it, with an attenuation factor given by the covariance matrix $\Sigma_k$. \\
At a formal level, the  multi-population MFG system corresponding to the previous problem  is given   by 
\begin{equation}\label{se_MFG}
	\left\{
	\begin{array}{ll}
		-\e  \Delta u_k(x)+ \half |D u_k(x)|^2+\l_k= \half (x-\mu_k)^t(\Sigma_k^{-1})^t\Sigma_k^{-1}(x-\mu_k) ,\quad& x\in\R^d, \\[8pt]
		\e    \Delta m_k(x)+ \diver(m_k(x)D u_k(x) )=0,&x\in\R^d, \\[8pt]
		\a_k=\int_{\R^d }\g_k(x) f(x)dx,\\[8pt] 
		m_k\ge 0,\,\int_{\R^d} m_k(x)dx=1,    u_k( \m_k)=0,
	\end{array}
	\right.
\end{equation}
for   $k=1, \dots, K$, where $\g_k$, $\m_k$, $\Sigma_k$ are defined as in \eqref{emmfg_resp}-\eqref{emmfg_var}.\\
\noindent

\begin{definition}\label{d:consistency}
	We say that a Gaussian mixture  $m(x)=\sum_{k=1}^K \b_k\cN(x|\n_k,T_k)$, $x \in \mathbb{R}^d$,
	is consistent with the data set if  
	\begin{equation}\label{se_cons}
		\left\{
		\begin{array}{ll}
			\n_k=\dfrac{\int_{\R^d} x \g_k(x)f(x)dx}{\int_{\R^d}  \g_k(x)f(x)dx}\\[16pt]
			T_k=\e\,\dfrac{\int_{\R^d } (x-  \n_k)(x-\n_k)^t\g_k(x)f(x)dx}{\int_{\R^d } \g_k(x)f(x)dx},\\[16pt]
			\b_k=\int_{\R^d}  \g_k(x)f(x)dx
		\end{array} 
		\right.
	\end{equation}
	where the responsabilities $\g_k$ are defined as in \eqref{emmfg_resp} with $m_k(x)=\cN(x|\n_k,T_k)$.	
\end{definition}
Note that $\e$ can be considered as an additional parameter of the model, that we can tune to adjust the covariance matrix of the function $m_k$.
\begin{proposition}
	Let $\e=1$ and $m$ be a Gaussian mixture consistent with the data set $f$. Then $m$ and $f$ have the same mean and covariance, i.e. 
	\begin{eqnarray}
		&\int_{\R^d } x m(x) dx = \int_{\R^d } x f(x) dx \label{mean}\\
		&\int_{\R^d } (x- \mu)(x-\mu)^tm(x) dx = \int_{\R^d } (x- \mu)(x-\mu)^t f(x) dx \label{variance}
	\end{eqnarray} 
\end{proposition}
\begin{proof}
	To prove \eqref{mean}, it suffices to apply the consistency conditions \eqref{se_cons} to   \eqref{emmfg_mix} and recall  that $\sum_{k=1}^{K}  \gamma_k (x) = 1$ to get 
	\begin{align*} 
		\int_{\R^d }xm(x)dx&=\sum_{k=1}^K\a_k\int_{\R^d}x m_k(x)dx=\sum_{k=1}^{K}\a_k\frac{\int_{\R^d} \,x \g_k(x)f(x)dx}{\int_{\R^d} \g_k(x)f(x)dx}\\
		&=\sum_{k=1}^{K}\int_{\R^d} \,x \g_k(x)f(x)dx  =\int_{\R^d }x f(x)dx.
	\end{align*} 
	To prove \eqref{variance}, we first claim that  
	\begin{equation}\label{eq_moment}
		\int_{\R^d } xx^t m(x) dx = \int_{\R^d} xx^t f(x) dx.
	\end{equation} 
	Indeed,   multiplying by $\alpha_k=\int_{\mathbb{R}^d}\g_k(x)f(x)dx $  the    consistency conditions \eqref{se_cons}, recalling  that $\sum_{k=1}^{K}  \gamma_k (x) = 1$ and summing over $k$, by the identity
	\[
	\sum_{k=1}^{K} \alpha_k \int_{\mathbb{R}^d}(x- \mu_k)(x-\mu_k)^t m_k (x) dx = \sum_{k=1}^{K} \alpha_k \dfrac{\int_{\R^d }(x - \mu_k)(x-\mu_k)^t \gamma_k (x) f(x) dx}{\int_{\mathbb{R}^d}\g_k(x)f(x)dx}, 
	\]
	we easily get 
	\begin{equation*}
		\int_{\R } xx^t m(x) dx - \sum_{k=1}^{K}\alpha_k \mu_k\mu_k^t= \int_{\R } xx^t f(x) dx - \sum_{k=1}^{K}\alpha_k\mu_k\mu_k^t \label{partialresul}
	\end{equation*}
	and therefore the claim.	
	Combining \eqref{eq_moment}  with the identity \eqref{mean}, we   get \eqref{variance}.
\end{proof}
In the next result, we show that if there exists a mixture   of  Gaussian  densities   consistent with the data set $f$,  then it solves   \eqref{se_MFG} for appropriate   $(u_k, \l_k)$, $k=1,\dots,K$.

\begin{proposition}\label{p:cons_MFG}
	Let $m$ be a mixture of Gaussian densities, i.e. $m(x)=\sum_{k=1}^K \b_k \cN(x|\nu_k, T_k)$  for $x \in \mathbb{R}^d$, consistent with the data set $f$.
	Then,   the family of quadruples  $(u_k, \lambda_k, m_k, \b_k)$, $k=1,\dots,K$, with 
	\begin{align*}
		&u_k(x)=\frac{\e}{2}(x-\m_k)^tT^{-1}_k(x-\m_k),\\
		&\lambda_k=\e^2 \Tr(T_k^{-1}),\\
		&m_k(x)=\cN(x|\m_k,T_k)\\
		&\b_k=\int_{\R^d}\gamma_k(x)f(x)dx
	\end{align*} 
	is a  solution of \eqref{se_MFG}.
\end{proposition}
\begin{proof} 
	Consider a component of the mixture $m_k$ which is a Gaussian density function $\cN(x|\nu_k,T_k)= \dfrac{1}{(2 \pi)^{\frac{d}{2}}\vert T_k \vert^{\half}}e^{-\half (x-\nu_k)^tT_k^{-1}(x-\nu_k)}$. Recall that $(x-\nu_k) \in \R^d$ and $T_k^{-1} \in \mathbb{R}^{d \times d}$ is the inverse of the covariance matrix $T_k$. 
	We have 
	\begin{equation*}
	\begin{cases}
	D m_k(x)&= -m_k(x) T_k^{-1}(x-\nu_k)\\
	D^2m_k(x)&=m_k(x)\big[ (T_k^{-1}(x-\nu_k))(T_k^{-1}(x-\nu_k))^t- T_k^{-1}\big]\\
	\Delta m_k(x) &= m_k(x)\big[(T_k^{-1}(x-\nu_k))^t(T_k^{-1}(x-\nu_k))-\Tr(T_k^{-1})\big]
	\end{cases}
	\end{equation*}
	 where $x \in \mathbb{R}^d$, $D^2$ denotes the Hessian and $\Tr$ the trace. 
	Furthermore, given a matrix $B_k \in \mathbb{R}^{d \times d}$ and assuming that 
	\begin{equation*}
		u_k(x)=\half (x-\rho_k)^tB_k(x-\rho_k),
	\end{equation*}
	we have
	$D u_k(x)= B_k(x-\rho_k)$ and $	\Delta u_k(x)=\Tr B_k$.
	Substituing $m_k$ and $u_k$ as above in the second equation  of \eqref{se_MFG}, we get
	\[
	\e \big[(T_k^{-1}(x-\nu_k))^t(T_k^{-1}(x-\nu_k))-\Tr(T_k^{-1})\big]-(T_k^{-1}(x-\nu_k))^t(B_k(x-\rho_k)) + \Tr(B_k))=0
	\]
	which is satisfied for $\rho_k=\nu_k$ and $B_k= \e T_k^{-1}$. Hence we get 
	\begin{equation}\label{se_u}
		u_k(x)= \dfrac{\e}{2}(x-\nu_k)^tT_k^{-1}(x_k-\nu_k) +C_k
	\end{equation}
	for some $C_k\in\R$. 
	Replacing  	$u_k$, given by \eqref{se_u},  in   the Hamilton-Jacobi equation in \eqref{se_MFG},   we have  
	\[
	-\e^2 \Tr (T_k^{-1})+ \dfrac{\e^2}{2}\vert (x-\nu_k)^t T_k^{-1}(x-\nu_k)\vert +\lambda_k=F_k(x, m_k, m)
	\]
	Recalling that $F_k(x, m_k, m)= \half (x-\mu_k)^t(\Sigma_k^{-1})^t\Sigma_k^{-1}(x-\mu_k)$ and that, by \eqref{se_cons},  
	$\n_k=\m_k$ and $T_k=\e \Sigma_k$, we get that the previous equation is satisfied for
	\[
	\l_k=\e^2 \Tr(T_k^{-1}). 
	\]	
	Taking into account the normalization condition $ u_k(\m_k)=0$, we finally get $C_k=0$ and then the result.
\end{proof}
We now prove a reverse result, i.e. a solution of the MFG system \eqref{se_MFG} gives a mixture of Gaussian densities consistent with the data set $f$.\\

\begin{proposition}\label{t:MFG_cons}
	Let $\{(u_k,\l_k,m_k,\a_k)\}_{k=1}^K$ be a solution of the  MFG system \eqref{se_MFG}. Then $m(x)= \sum_{k=1}^{K}\alpha_k m_k(x), \ x \in \mathbb{R}^d$ is a Gaussian mixture consistent with the data set $f$.
\end{proposition}
\begin{proof}
	Firstly, fixed $k=1, \dots, K$,  consider the Hamilton-Jacobi equation
	\begin{equation}\label{se_HJ}
		-\e  \Delta u_k(x)+ \half |D u_k(x)|^2+\l_k= \half (x-\mu_k)^t(\Sigma_k^{-1})^t\Sigma_k^{-1}(x-\mu_k) 
	\end{equation}
	with the normalization condition $u_k(\mu_k)=0$.	
	By \cite{bardi},   the unique solution  of the previous equation  is given  by the couple $(u_k,\l_k)=(\half(x-\mu_k)^t\Sigma_k^{-1}(x-\mu_k),\e\Tr(\Sigma_k^{-1}))$. Replacing   $D u_k(x)=\Sigma_k^{-1}(x-\m_k)$ in the Fokker-Planck equation, we obtain  the equation
	\begin{equation}\label{se_FP}
		\e   \Delta m_k(x)+ D m_k(x) \Sigma_k^{-1}(x-\mu_k)+ m_k(x) \Tr(\Sigma_k^{-1})=0.
	\end{equation}
	Taking into account the normalization conditions $m_k\ge 0$, $\int m_kdx=1$, we have that the solution of
	\eqref{se_FP} is given by the Gaussian density	$m_k(x)=\cN(x|\n_k,T_k)$,
	where, by \eqref{se_FP}, the constants $\n_k,T_k$ have to satisfy  the identity
	\[ \e \big((T_k^{-1}(x-\nu_k))^t(T_k^{-1})(x-\nu_k)-\Tr(T_k^{-1})\big)+\]\[
	- (T_k^{-1}(x-\nu_k))^t(\Sigma_k^{-1}(x-\mu_k))+\Tr(\Sigma_k^{-1})=0 \]
	Hence $m_k$ is a solution of \eqref{se_FP} if $\nu_k=\m_k$, $T_k= \e\Sigma_k$.
	If we set $m(x)= \sum_{k=1}^{K}\alpha_k m_k(x)$, with $m_k(x)=\cN(x|\m_k,\e \Sigma_k)$ and $\a_k$ as in \eqref{se_MFG},  we get a Gaussian mixture  compatible with the data set $f$.
\end{proof}
Lastly, we state  an existence result for    solutions to \eqref{se_MFG} (since the proof is similar to the one of Theorem \ref{theorem}, we omit the details).
\begin{proposition}
	There exists a solution $\{u_k, \lambda_k, m_k, \alpha_k\}_{k=1}^{K}$ to \eqref{se_MFG}.
\end{proposition}
We also observe that, as for \eqref{kmfg_MFG}, it is not reasonable to expect uniqueness for the solution of \eqref{se_MFG}.
\begin{remark}\label{se_general}
	In $1$-dimensional case, the coupling term is
	\begin{equation}\label{rem:0}
		F(x,m_k, m)=\half\left|\dfrac{x-\m_k}{\s_k^2}\right|^2.
	\end{equation}
	where
	\begin{equation}
		\mu_k=\frac{\int_{\R}x\gamma_{k}(x)f(x)dx}{\int_{\R} \gamma_{k}(x)f(x)dx}, \quad  
		\sigma_k^2=\,\dfrac{\int_{\R } (x-  \mu_k)^2\g_k(x)f(x)dx}{\int_{\R } \g_k(x)f(x)dx}.
		\label{rem:12}
	\end{equation}
	A Gaussian mixture $m(x)=\sum_{k=1}^K \b_k\cN(x|\nu_k,\t_k)$ is said to be consistent with the data set, in sense of Definition \ref{d:consistency}, if $\n_k=\m_k$, $\t^2_k=\e \sigma^2_k$, where $\mu_k$ and $\sigma_k$ defined as in \eqref{rem:12}, and $\b_k=\int_{\R}\gamma_{k}(x)f(x)dx$.\\
\end{remark}
A general class of  linear-quadratic MFG systems admitting a Gaussian distribution as solution of the Fokker-Planck equation was studied in \cite{bardi}.
\begin{remark}
	If we consider   a coupling cost $F_k$ in \eqref{emmfg_cost} independent of $\Sigma_k$, i.e. $\Sigma_k=\sigma^2 I$, where $I$ is the identity matrix, we have $ F_k(x,m_k, m)=\dfrac{1}{2\sigma^4} (x-\mu_k)^t(x-\mu_k)$. Then, the solution of the MFG system \eqref{se_MFG} gives a mixture of Gaussian densities
	$m(x)=\sum_{k=1}^K\a_k\cN(x|\m_k,\e \sigma^2 I )$.\\
	Arguing formally and sending $\sigma^2\to 0$,  the responsibility $\g_k(x)$, 	corresponding to  the index $k$   minimizing the distance of $x$ from the barycentre $\m_k$,  tends to $1$, while the other responsibilities $\g_j(x)$, for $j\neq k$, tend   to $0$.  Set $S_k=\{x\in\R^d:\lim_{\sigma^2\to 0} \g_k(x)=1\}$ and observe that these sets coincide with the clusters defined in \eqref{kmfg_cluster}   for the MFG K-means model.
	Moreover the consistency condition for the mean in  \eqref{se_cons} converges to
	\[\int x m_k(x)dx=\frac{\int_{S_k} xf(x)dx}{\int_{S_k} f(x)dx}\]
	which coincides with \eqref{kmfg_regle}. Hence the MFG K-means model can be interpreted, at least formally, as the limit of the MFG EM model for $\sigma^2\to 0$ (we plan to do a more rigorous analysis of this observation in the future).\\
	Sending \ $\epsilon \to 0^+$ in \eqref{se_MFG} has a different effect in the model. Indeed, in this case, the
	problem degenerate in a first order MFG system and the finite  mixture $m$ tends to a convex combination of Dirac functions concentrated in the barycentres $\mu_k$.
\end{remark}
\begin{remark}
	The approach developed in this section is a soft clustering model as the classical EM algorithm, while the one in Section \ref{subsec:coron} is a hard clustering one, as the K-means algorithm. In both the models, the coupling terms   have the effect of aggregating the agents around the closer barycenter. But, while in the MFG K-means model, the clusters are circular, in the model discussed in this section, thanks to the additional parameters given by the covariance matrices, the clusters are ellipsoidal and therefore, in some situations,    give a better approximation of the data set. This advantage of the EM algorithm with respect to the K-means algorithm  is   well-known also in the classical setting discussed in Section \ref{sec:Algorithm}.\\
	We consider an ergodic MFG system, see  \eqref{se_MFG}, instead of a discounted MFG system as \eqref{kmfg_MFG}, but this does not change in an essential way the model. The choice of an ergodic system has been made in order to apply the algorithm developed in \cite{cc}.  We also observe that system \eqref{se_MFG} has the classical structure of a MFG system, where the Fokker-Planck equation is the adjoint of the linearized of the Hamilton-Jacobi-Bellman equation, while this property is not satisfied by system \eqref{kmfg_MFG}.
	
\end{remark}

\section{The general MFG EM model}\label{sec:general_problem}
In this section, we formulate a general MFG model for a finite mixture of probability densities. \par
We consider a smooth, bounded domain $\Omega$ containing 
the support of the  data set $f$ which satisfies the conditions in \eqref{mfg_density}.  Given a mixture of density functions as in \eqref{emmfg_mix} and defined the corresponding responsibilities $\g_k$ as in \eqref{emmfg_resp}, we consider the following multi-population MFG system

 \begin{equation}\label{gen_MFG}
\left\{
\begin{array}{lll}
-\e  \Delta u_k(x)+ H_k(x,Du_k(x))+\l_k=  F_k(x,m_k,m),\quad& x\in \Omega,& \textit{(HJ)} \\[8pt]
\e  \Delta m_k(x)+\diver (m_k(x)D_pH_k(x,Du_k(x)) )=0,&x\in \Omega, &\textit{(FP)}\\[8pt]
\pd_n u_k(x)=0,&x\in \partial \Omega,&\textit{(HJn)}\\[8pt]
\e\pd_n m_k(x)+m_k(x)D_pH_k(x,Du_k(x))\cdot \vec{n} =0,&x\in \partial \Omega,&\textit{(FPn)}\\[8pt]
\a_k=\int_\Omega    \g_k(x)  f(x)dx,  & &\textit{(COEFF)}\\[8pt] 
 m_k\ge 0,\,\int_\Omega m_k(x)dx=1, \int_\Omega u_k( x)dx=0,& &\textit{(NORM)}
\end{array}
\right.
\end{equation}
for $k=1,\dots,K$, where $\vec{n}$ is the outward normal to the boundary of $\Omega$, $\pd_n$ denotes the  normal derivative and $\g_k(x)=\a_k m_k(x)/m(x)$, $k=1,\dots,K$ are the responsibilities. The boundary conditions in \eqref{gen_MFG} are given by the Neumann condition (HJn) for   $u_k$ and its dual (FPn) for $m_k$. Moreover, the normalization conditions (NORM) are imposed since $u_k$ is defined up to a constant and $m_k$ represents a probability densiy function.
	  \par
In the rest of the paper, we will make the following assumptions:
\begin{itemize}
	\item[{\bf (H1)}] $\Omega$ is a $C^2$ bounded domain of $\R^d$;
	\item[{\bf (H2)}] the Hamiltonians $H_k$, $k=1,\dots,K$,    are of the form
	\[H_k(x,p)=R_k|p|^r-H^k_{0}(x),\]
	with $R_k>0$, $r> 1$, $H^k_{0}\in C^2(\Omega)$, $\pd_n H^k_{0}(x)\ge 0$ for $x\in \pd \Omega$;
	\item[{\bf (H3)}] The coupling costs $F_k:\Omega\times\R\times\R^{d\times d}\to \R$ are a continuous, uniformly bounded functions of the form  \begin{equation}\label{simple_coupling}
	F_k(x, m_k, m)=\mathcal{F}_k(x,\m_k,\Sigma_k),
	\end{equation} 
	where 
	\begin{align}
	&\mu_k=\int_{\Omega}x\frac{m_{k}(x)}{m(x)}f(x)dx\label{hyp:1},\\  
	&\Sigma_k=\int_{\Omega}(x-\mu_k)(x-\mu_k)^t\frac{m_{k}(x)}{m(x)}f(x)dx.
	\label{hyp:2}
	\end{align}
	 Moreover, for $\m$, $\Sigma$ fixed,   $\mathcal{F}_k(\cdot,\m ,\Sigma)\in W^{1,\infty}(\Omega)$.
\end{itemize}
\begin{remark}
	If $\a_k=\int_\Omega \gamma_k(x) f(x)dx\neq 0$, then,  
	 \[\mu_k=\frac{1}{\a_k}\int_{\Omega}x\frac{\a_k m_{k}(x)}{m(x)}f(x)dx= \frac{\int_{\Omega}x\gamma_{k}(x)f(x)dx}{\alpha_k}=\frac{\int_{\Omega}x\gamma_{k}(x)f(x)dx}{\int_\Omega \gamma_k(x) f(x)dx},\]
which has to be compared with the definition of $\m_k$ in  \eqref{emmfg_mean} (similarly  for $\Sigma_k$). As we will see in the proof of the existence theorem, since $m(x)>0$ for all $x\in \Omega$, definitions \eqref{hyp:1}-\eqref{hyp:2} avoid the problem of having $\a_k=0$ in the definition of $\mu_k$, $\Sigma_k$.
\end{remark}   

\begin{remark}
	In view of different  applications from  Cluster Analysis, it is possible to consider more general Hamiltonians and coupling terms. For example, coupling terms $F_k$ involving the Kullback-Leibler divergence $$m_k \ln\left( \frac{q_k }{m_k }\right),$$ 
	where $q_k$ is a function of the data set $f$.
	 But to maintain the parallelism with classical EM algorithm, we prefer to restrict to  the simpler case \eqref{simple_coupling}.
\end{remark}

We   now  give the definition of solution for the  MFG system \eqref{gen_MFG}. Note that the coefficients of the mixture $\a_k$ are part of the unknowns of the system.
\begin{definition}\label{def_sol}
A solution of \eqref{gen_MFG} is a family of $K$  quadruples  $(u_k,\l_k,m_k,\a_k)$  such that $u_k\in C^2(\overline \Omega)$, $\l_k\in\R$, $m_k\in W^{1,2}(\Omega)$, $\a_k\in[0,1]$ 	 and
\begin{itemize}
\item[(i)] $(u_k,\l_k)$ satisfies 	(HJ)-(HJn) in point-wise sense.
\item[(ii)] $m_k$ satisfies (FP)-(FPn) in weak sense, i.e.
\[\e \int_\Omega  Dm_k(x)\cdot D\phi(x)dx+ \int_\Omega m_k(x)D_pH_k(x,Du_k(x))\cdot D\phi(x) dx=0\]
   for all $\phi\in W^{1,2}(\Omega)$.
\item[(iii)]  The coefficients $\a_k$ satisfy (COEFF) (note that this implies  $\a_k\in [0,1]$ and $\sum_{k=1}^K\a_k=1$).
\item[(iv)] The normalization conditions (NORM)  are satisfied.
\end{itemize}
\end{definition}
We are going to prove  existence of a solution to \eqref{gen_MFG}. The proof is similar to the one of \cite[Theorem 4]{cirant}, the main difference is   the presence of the additional unknowns given by the coefficients $\a_k$.
\begin{theorem}\label{theorem}
	Assume {\bf (H1)-(H3)}. Then, there exists a solution $(u_k, \lambda_k, m_k, \alpha_k) $, $k=1,\dots,K$, of \eqref{gen_MFG} in the sense of Definition \ref{def_sol}. Moreover $u_k\in C^2(\ov \Omega)$, $m_k\in W^{1,p}(\Omega)$ for any $p\ge 1$ and $m_k\ge \delta>0$ for some constant $\delta$.
\end{theorem}
\begin{proof}
Fixed $p>d$, we define the sets
\begin{align*}
&\mathcal{P}= \{m \in L^1(\Omega)\  \text{s.t.} \ \int_{\Omega} m(x)dx=1 \}\\
& \cD= \mathcal{P}\cap \{m  \in W^{1, \infty}(\Omega) \ \text{s.t.} \ \vert \vert m \vert \vert_{W^{1, p}(\Omega) }\leq  C, m(x) \geq \delta >0 \},
\end{align*}
where $C$, $\d$ are two constants to be fixed later, and we consider the set
\begin{equation}\label{set_fix}
\begin{split}
\mathcal{C}= \{ (\a,m)=&(\alpha_1, \dots, \alpha_K, m_1, \dots, m_K): \,\a_k \in  [0,1], \\
&\ m_k  \in  \cD, \ k=1, \dots, K,\,  \sum_{k=1}^{K} \alpha_k=1\}.
\end{split}
\end{equation}
It is easy to see that the set $\mathcal{C}$ is convex. Moreover, it is compact with respect to the standard topology of $\R\times\dots\times\R\times C(\ov \Omega)\times\dots\times C(\ov \Omega)$ since $\a_k\in [0,1]$, $k=1, \dots, K$,  and the set $\cD$ is compactly embedded in $C^{0, \b}(\ov {\Omega})$ with $\b \leq (p-d)/p$ (see Theorem 7.26 in \cite{Gilbarg}).\\
We define the map $\Psi : \mathcal{C}\to \mathcal {C}$ which, to a  vector	$(\beta,\eta)=(\beta_1, \dots, \beta_K, \eta_1, \dots, \eta_K)$ associates the vector $(\alpha,m)=(\alpha_1, \dots, \alpha_K, m_1, \dots, m_K) $   defined in the following way:\\
{\it Step (i):} 	 
Given $(\beta, \eta) \in \mathcal{C}$,  we define the responsibilities  
\begin{equation}\label{exMFG_0}
\gamma_k(x) = \dfrac{\beta_k \eta_k(x)}{\sum_{k=1}^K\beta_k\eta_k (x)}, \quad   k=1, \dots, K.	
\end{equation}
Note the responsibilities are well defined since, by $\eta_k\ge \d>0$ and $\sum_{k=1}^K\b_k=1$, we get $\sum_{k=1}^K\beta_k\eta_k (x)\ge \d$ for $x\in \Omega$. \par
{\it Step (ii):} 
Given $\gamma_k(x)$ as in {\it Step (i)}, we define $\mu_k$, $\Sigma_k$ as in \eqref{hyp:1}-\eqref{hyp:2} and we consider the Hamilton-Jacobi-Bellman equations
	\begin{equation}\label{exMFG_1}
	\left\{
	\begin{array}{ll}
	-\e\Delta u_k(x)+ H_k(x,Du_k(x))+ \lambda_k=\mathcal{F}_k(x,\mu_k,\Sigma_k),\quad& x\in \Omega, \\[6pt]
	\partial_n u_k(x)=0 \quad & x \in \partial \Omega,\\[6pt]
	\int_{\Omega} u_k(x)dx=0,
	\end{array}
	\right.
	\end{equation}
for $k=1,\dots,K$. Note that the previous  systems  are independent of  each other. By \cite[Theorem II.1]{Lions}, \eqref{exMFG_1}
	admits a unique solution $( u_k,  \lambda_k) \in C^2(\bar \Omega) \times \mathbb{R}$ with
\begin{equation}\label{exMFG_2}
\vert \lambda_k \vert \leq    C, \qquad \Vert u_k \Vert_{W^{1, \infty}(\Omega)}\leq C
\end{equation}
for some positive constant $C$  which depends only on $\Omega$, $H$ and  $\Vert \mathcal{F}_k(\cdot,\mu_k,\Sigma_k) \Vert_{L^{\infty}}$.
\par
{\it Step (iii):} 
Given $u_k$ as in {\it Step (ii)}, we consider the problems
	\begin{equation}\label{exMFG_3}
	\left\{
	\begin{array}{ll}
	\e\Delta m_k(x)+\diver (m_k(x)D_p   H_k (x,Du_k(x) )=0 & x \in \Omega\\[6pt]
	\e \partial_n m_k(x)+ m_k(x)D_pH_k(x,Du_k(x))\cdot \vec{n}(x)=0& x\in \partial \Omega \\[6pt]
	m_k\ge 0,\,\int_\Omega m_k(x)dx=1
	\end{array}
	\right.
	\end{equation}
for $k=1,\dots,K$. Since   \eqref{exMFG_2} implies that $| D_pH_k(x,Du_k(x))|$  is uniformly bounded in $\Omega$,
by   Theorems II.4.4, II.4.5, II.4.7 in \cite{Bensoussan}, it follows that problem \eqref{exMFG_3}
admits a unique weak solution $ m_k\in W^{1,2}(\Omega)$. Moreover, $m_k \in W^{1, p}(\Omega)$ for all $p \geq 1$, $m_k \in C(\ov{\Omega})$ and  
\begin{align}
\vert \vert m_k \vert \vert _{W^{1, p}(\Omega)} \leq \hat{C},\label{exMFG_5}\\
\hat\delta\leq m_k(x) \leq \hat\delta^{-1} \quad \forall x \in \Omega,\label{exMFG_4}
\end{align}
 for   constants $\hat C$, $\hat\d>0$ which  depend   only on $\Vert  DH_k(x,Du_k(x))  \Vert_{L^{\infty}(\Omega)}$ and therefore, recalling \eqref{exMFG_2},    on $\Omega$, $H$ and  $\Vert \mathcal{F}_k \Vert_{L^{\infty}}$. 
Hence, in \eqref{set_fix}, fixing $C$ larger that $\hat C$ in \eqref{exMFG_5} and $\d$ smaller than $\hat \d$ in \eqref{exMFG_4}, we conclude $m_k \in \mathcal{D}$, \mbox{$k=1,\dots,K$.} \par
Furthermore, we define  
\begin{equation}\label{exMFG_6}
	\alpha_k := \int_{\Omega} \gamma_k (x) f(x) dx.
\end{equation}
Since $0\le \gamma_k (x)\le 1$   and $\sum_{k=1}^K\gamma_k (x)=1$ for $x\in \Omega$, it follows that $\alpha_k\in [0,1]$ and $\sum_{k=1}^K\alpha_k=1$.\par
We conclude that the map $\Psi$, which  to  the vector $(\eta,\beta)$ associates the vector $(\a,m)$ given by \eqref{exMFG_3}-\eqref{exMFG_6}, is well defined and maps the set $\mathcal{C}$ into itself.\par
We now show that the map $\Psi$ is continuous with respect to the $\R\times\dots\times\R\times C(\ov \Omega)\times\dots\times C(\ov \Omega)$   topology.	Let $(\beta^n, \eta^n)$ be a sequence in $\mathcal{C}$ such that    $\beta_k^n$ converges  $\beta_k$ in $\R$ and  $\eta_k^n $   converges  to $\eta_k$  uniformly in $\ov \Omega$, for all $k=1\dots,K$. Since   	$\sum_{k=1}^{K}\beta^n_k \eta^n_k $ converges  uniformly to $ \sum_{k=1}^{K} \beta_k \eta_k$  and $\eta^n_k,\eta_k\ge \delta>0$ in $\Omega$,  defined
$\gamma_k^n$, $\gamma_k$ as in \eqref{exMFG_0}, it follows that 
\begin{equation}\label{exMFG_7}
\gamma_k^n\to \gamma_k\quad \text{for $n\to \infty$, uniformly in $\ov \Omega$}.
\end{equation}
Then, by  $0\le \gamma_k^n,\gamma_k\le 1$,  \eqref{exMFG_7}  and the dominated convergence theorem, we have  
\begin{equation}\label{exMFG_8}
	\lim_{n\rightarrow \infty}\alpha_k^n  = \lim_{n\rightarrow \infty} \int_{\Omega}\gamma_k ^n(x) f(x) dx= \int_{\Omega} \gamma_k(x) f(x) dx= \alpha_k.
\end{equation}
Moreover, since $\sum_{k=1}^{K}\beta^n_k \eta^n_k $, $\sum_{k=1}^{K}\beta _k \eta_k\ge \d$ and $\eta^n_k$ converges to $\eta_k$ uniformly, we have
\begin{align*}
\dfrac{\eta_k^n}{\sum_{k=1}^{K}\beta^n_k \eta^n_k } \longrightarrow \dfrac{\eta_k}{\sum_{k=1}^{K}\beta_k \eta_k },\quad \text{for $n\to \infty$, uniformly in $\overline\Omega$}.
\end{align*} 
Recalling   definitions \eqref{hyp:1}-\eqref{hyp:2}, it follows that
	\begin{equation*}
 \lim_{n\rightarrow \infty} \mu_k^n =
	 \lim_{n\rightarrow \infty} \int_{\Omega}\dfrac{x\eta_k^n(x)f(x) }{\sum_{k=1}^{K}\beta^n_k \eta^n_k (x)}dx= \int_{\Omega} \dfrac{x\eta_k(x)f(x) }{\sum_{k=1}^{K}\beta_k \eta_k (x)}dx= \mu_k
	\end{equation*}
and, similarly,    $\lim_{n\rightarrow \infty}\Sigma_k^n=\Sigma_k$.\\
It follows that  $\mathcal{F}_k(x, \mu_k^n, \Sigma_k^n)$ converges uniformly to $\mathcal{F}_k(x, \mu_k, \Sigma_k)$.  Denoted 
by $(u^n_k,\l^n_k)$ and $(u_k, \l_k)$, $k=1,\dots,K$, the solutions of \eqref{exMFG_1} corresponding respectively to 
$\mathcal{F}_k(x, \mu_k^n, \Sigma_k^n)$ and $\mathcal{F}_k(x, \mu_k, \Sigma_k)$,   by a standard stability result for \eqref{exMFG_1}, we  have that  $\l^n_k$ converges to $\l_k$ and $u^n_k$  converges $u_k$ uniformly in $\ov \Omega$. Moreover $D u_k^n\to Du_k$ locally uniformly in $\Omega$ (see \cite[Theorem 4 ]{cirant} for details).\\
Let  $m_k^n$, $m_k$ be the solutions of \eqref{exMFG_3} corresponding to $u_k^n$ and $u_k$. By the estimate \eqref{exMFG_5}, the functions $m_k^n$ are equi-bounded in $W^{1,p}$. By the local uniform convergence of   $D u_k^n$ to $Du_k$,  passing to the limit in
the weak formulation of \eqref{exMFG_3} yields that (possibly passing to a subsequence) that $m_k^n$ converges to
a  weak solution of the problem, which, by uniqueness, coincides with $m_k$.  Hence $m^n_k$ converges uniformly to $m_k$, $k=1,\dots,K$, and, recalling \eqref{exMFG_8}, we conclude that the map $\Psi$ is continuous. \par
 Applying the Schauder's Fixed Point Theorem to the map $\Psi:\mathcal{C}\to \mathcal{C}$, we get that there exists
$(\alpha,m)\in \mathcal{C}$ such that $\Psi(\alpha,m)=(\alpha,m)$, hence a solution to \eqref{gen_MFG}.
 \end{proof}


\section{Numerical approximation and examples}\label{sec:Numericalanalysis}
In this section, we present an algorithm for the cluster problem, based on the numerical approximation of \eqref{gen_MFG}. 
Our aim is to compute the finite mixture model \eqref{emmfg_mix} and the corresponding responsibilities $\g_k(x)$. 
We use an approach similar to the classical EM algorithm. Indeed, rather than solving the full problem \eqref{gen_MFG} in the coupled unknowns $(u_k,\lambda_k,m_k,\alpha_k)$ 
for $k=1,...,K$, we split it in the iterative resolution of $K$ independent sub-problems as follows.

Given an arbitrary initialization $\alpha^0_1, \dots, \alpha^0_k$, $m^0_1 \dots, m^0_k$ for the mixture model, at the $n^{th}$ iteration, we alternate two steps:
\begin{itemize}
	\item []\textbf{E-step.} For $k=1,\dots,K$, compute the new responsibilities, mixture coefficients, barycentres and covariance matrices
	\begin{equation}\label{Estep}
	\gamma^{n}_k(x)= \dfrac{\alpha_k^{n-1}m^{n-1}_k(x)}{m^{n-1}(x)}\,,\quad \a_k^n=\int_\Omega \g^n_k(x) f(x)dx\,,
	\end{equation}
	$$\mu_k^n=\int_{\Omega}x\frac{m_{k}^{n-1}(x)}{m^{n-1}(x)}f(x)dx\,,\quad\Sigma_k^n=\int_{\Omega}(x-\mu_k^n)(x-\mu_k^n)^t\frac{m_{k}^{n-1}(x)}{m^{n-1}(x)}f(x)dx\,.$$
	\item []\textbf{M-step.} For $k=1,\dots,K$ compute the new coupling cost $\mathcal{F}_k(x,\m^n_k,\Sigma^n_k)$ and solve the (independent) $K$     systems
	\begin{equation}\label{MFf_n1}
	\left\{
	\begin{array}{ll}
	-\e  \Delta u_k(x)+ H_k(x,Du_k(x))+\l_k= \mathcal{F}_k(x,\m^n_k,\Sigma^n_k)  ,\quad& x\in \Omega,  \\[8pt]
	\e  \Delta m_k(x)+\diver (m_k(x)D_pH_k(x,Du_k(x)) )=0,&x\in \Omega,  \\[8pt]
	\pd_n u_k(x)=0,&x\in \partial \Omega, \\[8pt]
	\e\pd_n m_k(x)+m_k(x)D_pH_k(x,Du_k(x))\cdot\vec{n} =0,&x\in \partial \Omega, \\[8pt]
	m_k\ge 0,\,\int_\Omega m_k(x)dx=1, \int_\Omega u_k( x)dx=0.&  
	\end{array}
	\right.
	\end{equation}
\end{itemize}
Note that the coupling between the mixture components is entirely embedded in the E-step, and this significantly reduces 
the computational efforts to perform the M-step, in particular the $K$ sub-problems \eqref{MFf_n1} can be solved in parallel. Moreover, for fixed $k$, we observe that the  Hamilton-Jacobi-Bellman equation  is decoupled from the corresponding Fokker-Planck equation. In principle, we could compute  first $u_k$ and then use it to compute $m_k$. Nevertheless, we prefer to solve the system \eqref{MFf_n1} as a whole, also in view of more general coupling costs, in which the dependence on $m_k$ is not frozen at the previous step as in \eqref{Estep}.\\ 
Let us now discuss, more in detail, the main ingredients for a practical implementation of the algorithm. 
First of all, we have to introduce a discretization of the domain $\Omega$. The choice of a structured or an unstructured grid clearly depends 
on the geometry of $\Omega$ and, in turn, on the nature of the data set represented by the distribution $f$. Typically $\Omega$ can be chosen as 
a simple hypercube in $\R^d$ 
containing the support of $f$, as in the tests shown below.\\
 Then, we need some quadrature rule for the approximation of the integrals appearing in the E-step, 
and also a suitable discretization of 
the MFG systems in the M-step, e.g. via finite differences or finite elements (\cite{acd,afd,cs}). Note that each system in \eqref{MFf_n1} consists of two non-linear PDEs 
in the three unknowns $(u_k,m_k,\lambda_k)$, plus 
boundary and normalization conditions. This leads to the building block of the algorithm, namely the numerical solution of a discrete non-linear system 
which is formally {\em overdetermined}, having more equations than unknowns. More precisely, if $\Omega$ is discretized with $N$ of degrees of freedom, 
then $(u_k,m_k,\lambda_k)$ corresponds to a vector of $2N+1$ unknowns, while \eqref{MFf_n1} gives $2N+2$ non-linear equations, subject to $N$ constraints representing 
the condition $m_k\ge 0$. Hence, by employing a Newton-like solver, we end up with non square linearized systems, whose solutions should be meant in a {least-squares} sense, 
as explained in \cite{cc}, where a {\em direct method} for MFG systems is proposed. 
We refer the reader to this paper for further details and the implementation of the method, which will be used in the following tests.

We conclude this discussion with some remarks on the initialization and the stopping criterion of the algorithm. 
As it is commonly used in the classical EM algorithm, we choose all the initial weights $\alpha^0_k=\frac{1}{K}$, and the densities $m^0_k$ given by Gaussian 
PDFs $\cN(x|\mu^0_k,I)$ with random means $\mu^0_k$. The convergence of the algorithm is really sensitive to this initial guess, and 
we observe a relevant speedup if the supports of the $m^0_k$ are as much as possible disjointed. Finally, we choose a tolerance $tol>0$ and we iterate on $n$ the EM-step 
until both $\max_{k=1,...,K}|\m^n_k-\m^{n-1}_k|<tol$ and $\max_{k=1,...,K}|\Sigma^n_k-\Sigma^{n-1}_k|<tol$. 
Note that also the M-step has its own tolerances for the convergence of the Newton solvers for the $K$ MFG systems. In principle, we can choose a single tolerance for the 
whole algorithm and alternate the E-step and M-step in a Gauss-Seidel fashion, namely updating the parameters \eqref{Estep} after a single Newton iteration for 
each system \eqref{MFf_n1}. Surprisingly, this does not improve the convergence and the proposed implementation with two nested loops 
is more effective. Indeed, we observed that, once each mixture component $m_k$ stabilizes around its current mean $\m_k$, in the next iterations it readily moves to 
the new barycentre with just small adjustments of the corresponding covariance matrix $\Sigma_k$. In practice, most of the computational efforts are concentrated on the 
first few iterations of the algorithm. 

Now, let us define the settings for the numerical experiments. We use a uniform structured grid to discretize the domain $\Omega$, and 
a simple rectangular quadrature rule for the E-step. Moreover, we employ standard finite differences for the discretization of the M-step, where we 
consider the particular case of a quadratic Hamiltonian as in \eqref{se_MFG} (see also Remark \ref{se_general}), and we set the diffusion 
coefficient $\varepsilon=1$.
The case of general Hamiltonians and coupling costs will be addressed in a more computational oriented paper. \\     

\noindent{\bf Test 1. } We consider a piece-wise constant distribution $f$ on $\Omega=[0,1]$, composed by three plateaux of different widths and heights, 
such that $\int_0^1 f(x)\,dx=1$. 
In Figure \ref{Test1}, we show 
the solution computed on a uniform grid of $N=201$ nodes, respectively for $K=1,2,3$. The thin line represents $f$, 
while the thick line represents the mixture $m=\sum_{k=1}^K \alpha_k m_k$. 
We clearly see how the mean and the variance of each $m_k$ adapt to the data, according to the 
given number $K$ of mixture components.\\
\begin{figure}[!t]
	\begin{center}
		\begin{tabular}{ccc}
			\includegraphics[width=.3\textwidth]{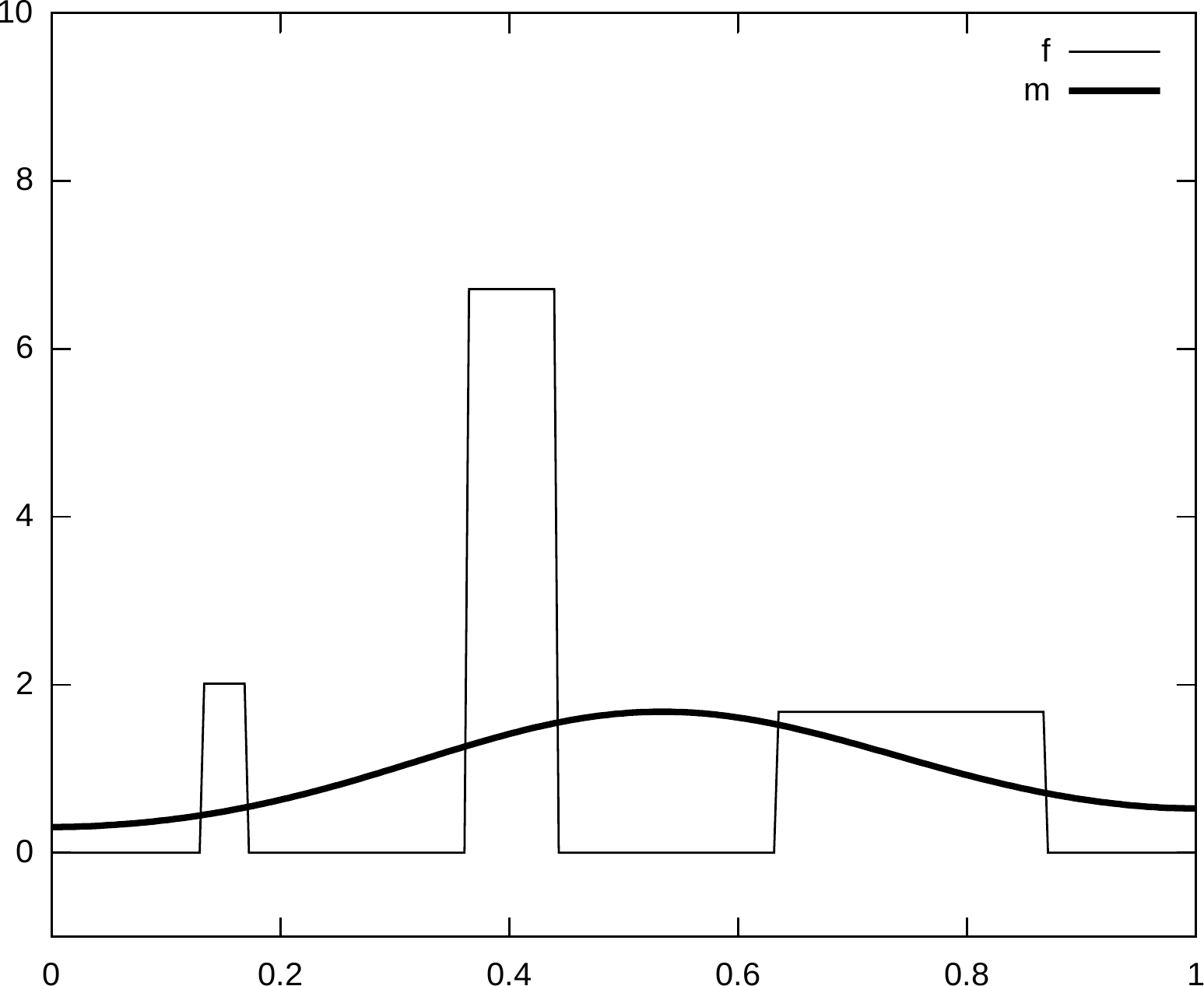}&
			\includegraphics[width=.3\textwidth]{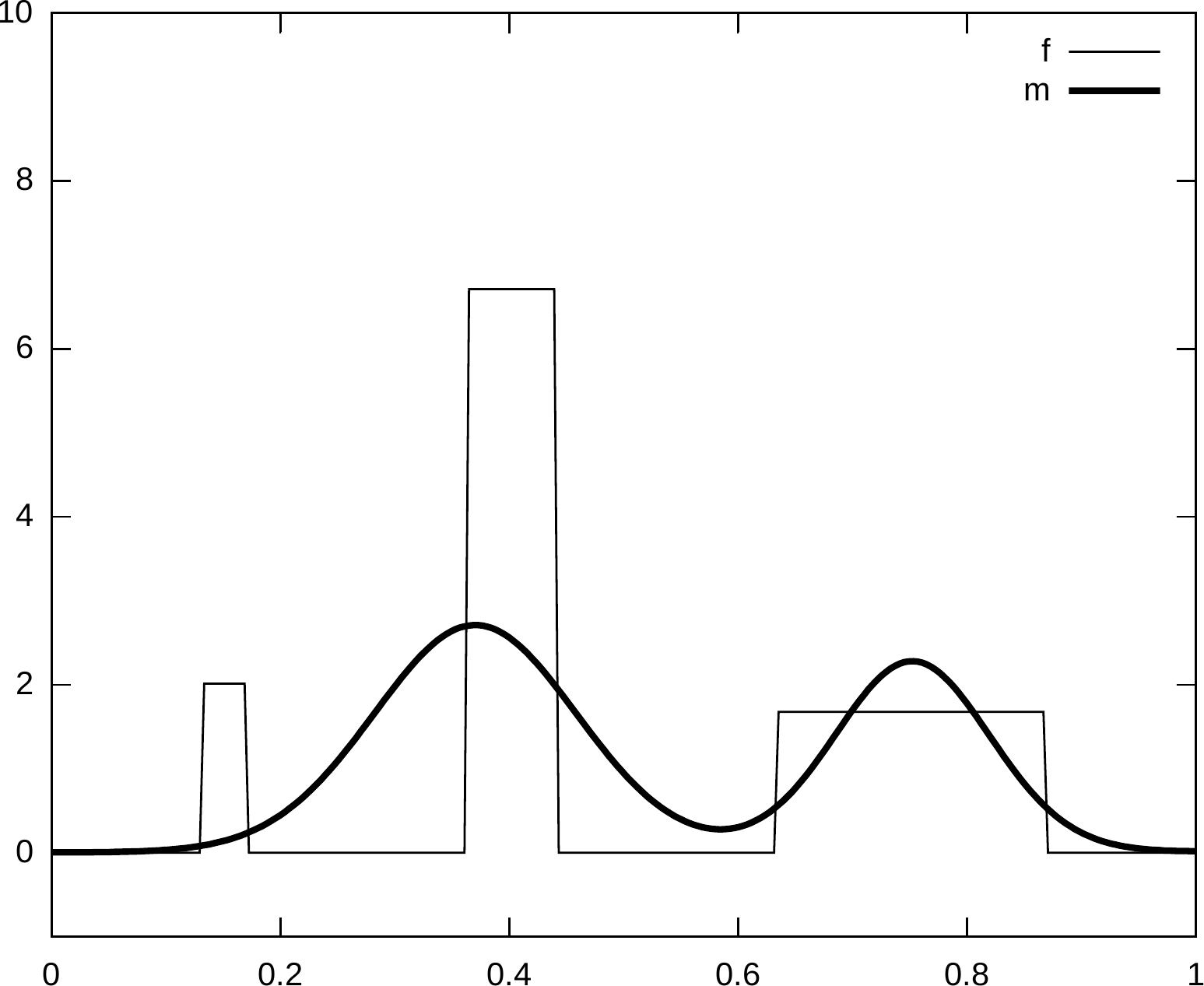}&
			\includegraphics[width=.3\textwidth]{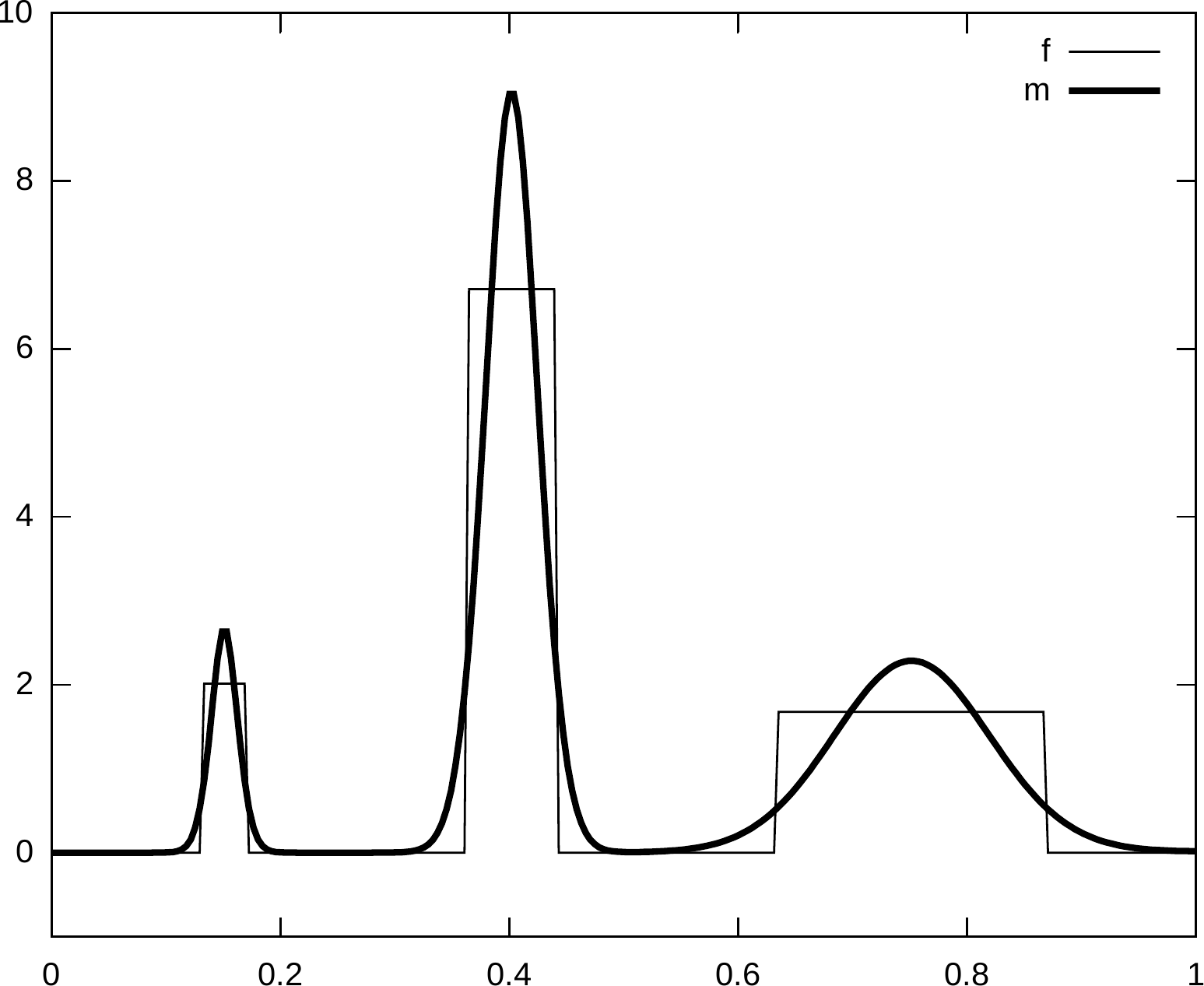}\\
			(a)&(b)&(c)
		\end{tabular}
	\end{center}
	\caption{Approximation of a piece-wise constant data distribution, for $K=1,2,3$ mixture components (from (a) to (c)).}\label{Test1}
\end{figure}

\noindent{\bf Test 2.} We consider the example of a distribution on $\Omega=[0,1]$ with an oscillatory behavior. Indeed, we define 
$f$ by suitably scaling and translating the function $\tilde f=x\sin(4\pi x)$ for $x\in[0,1]$, so that $f$ has compact support and $\int_0^1 f(x)dx=1$. 
In Figure \ref{Test2}, we show the solutions corresponding to $K=1,2,3,4$, again for a uniform grid of $N=201$ nodes. It is interesting to observe that the peaks of $f$ are sequentially approximated as the number $K$ of mixture components increases, according to their heights and the underlying masses.\\

\begin{figure}[!h]
	\begin{center}
		\begin{tabular}{cc}
			\includegraphics[width=.4\textwidth]{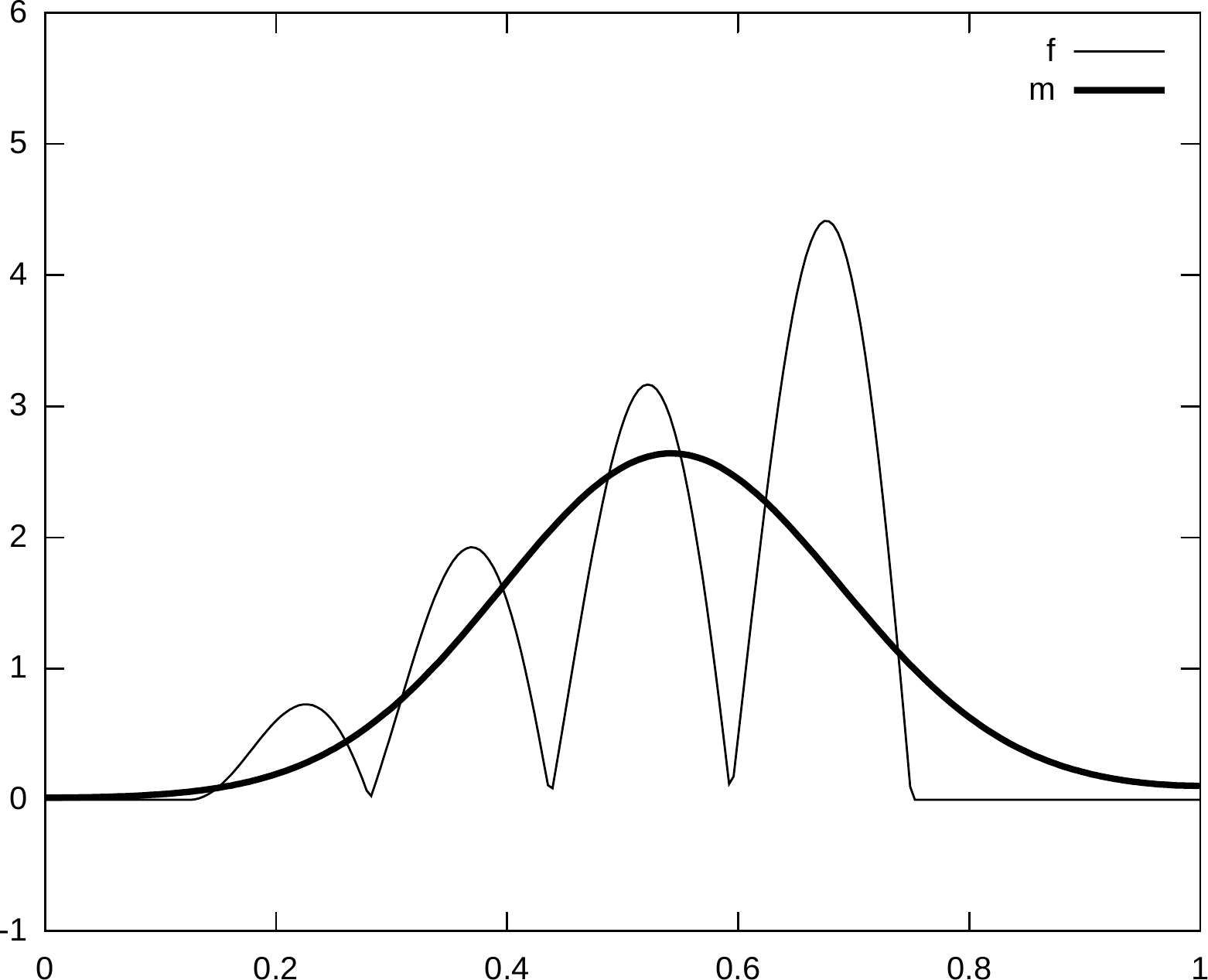}&
			\includegraphics[width=.4\textwidth]{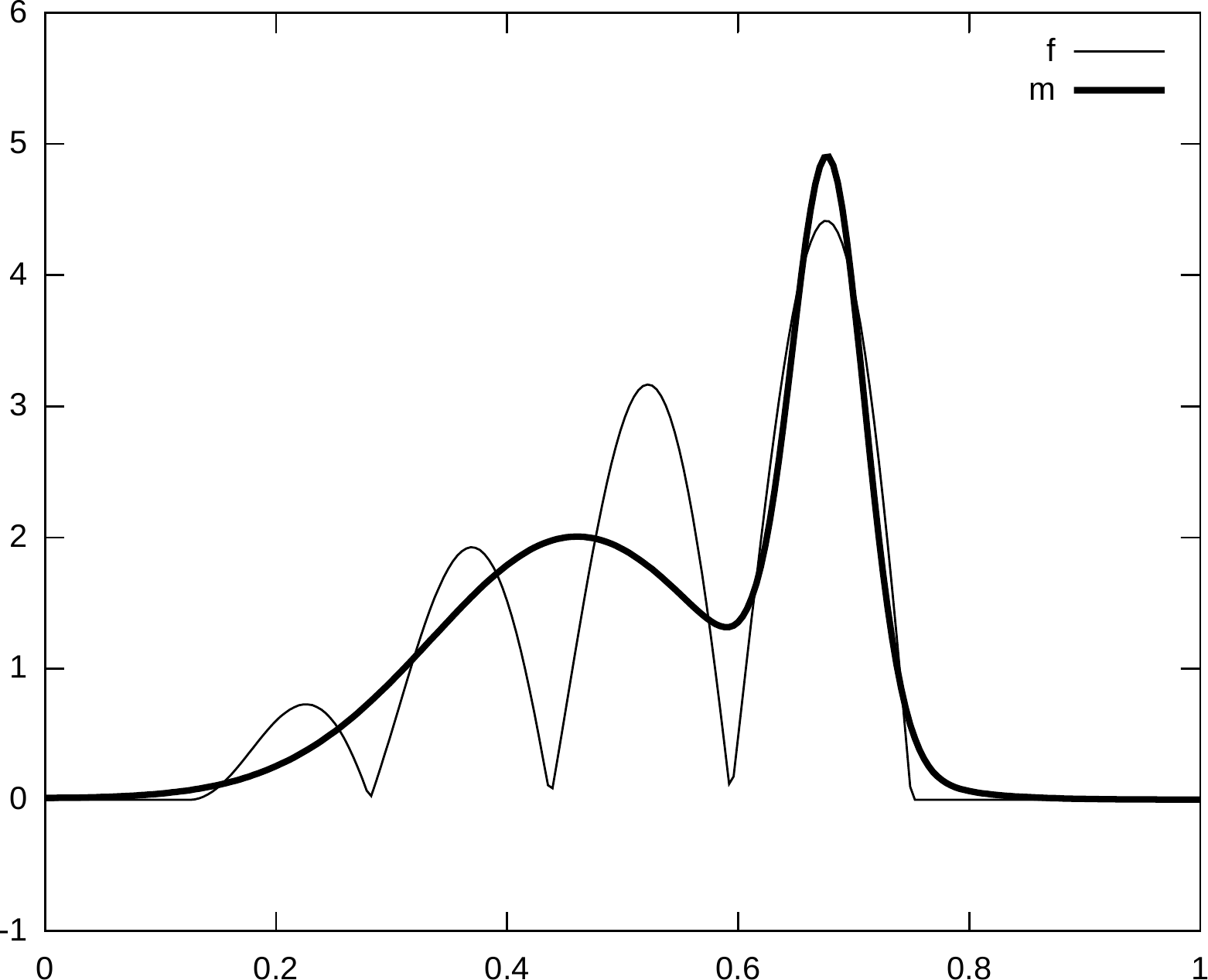}\\
			(a)&(b)\\
			\includegraphics[width=.4\textwidth]{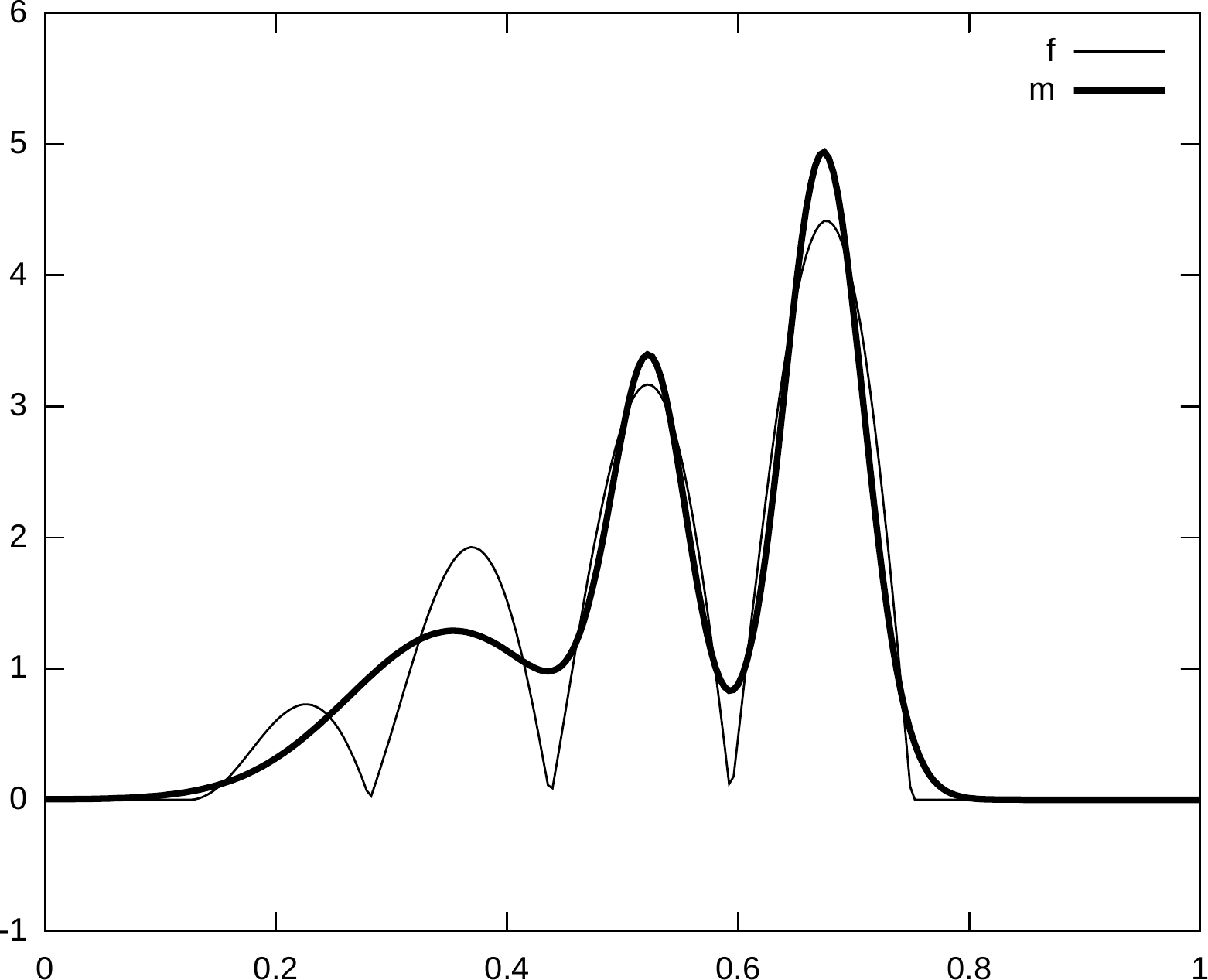}&
			\includegraphics[width=.4\textwidth]{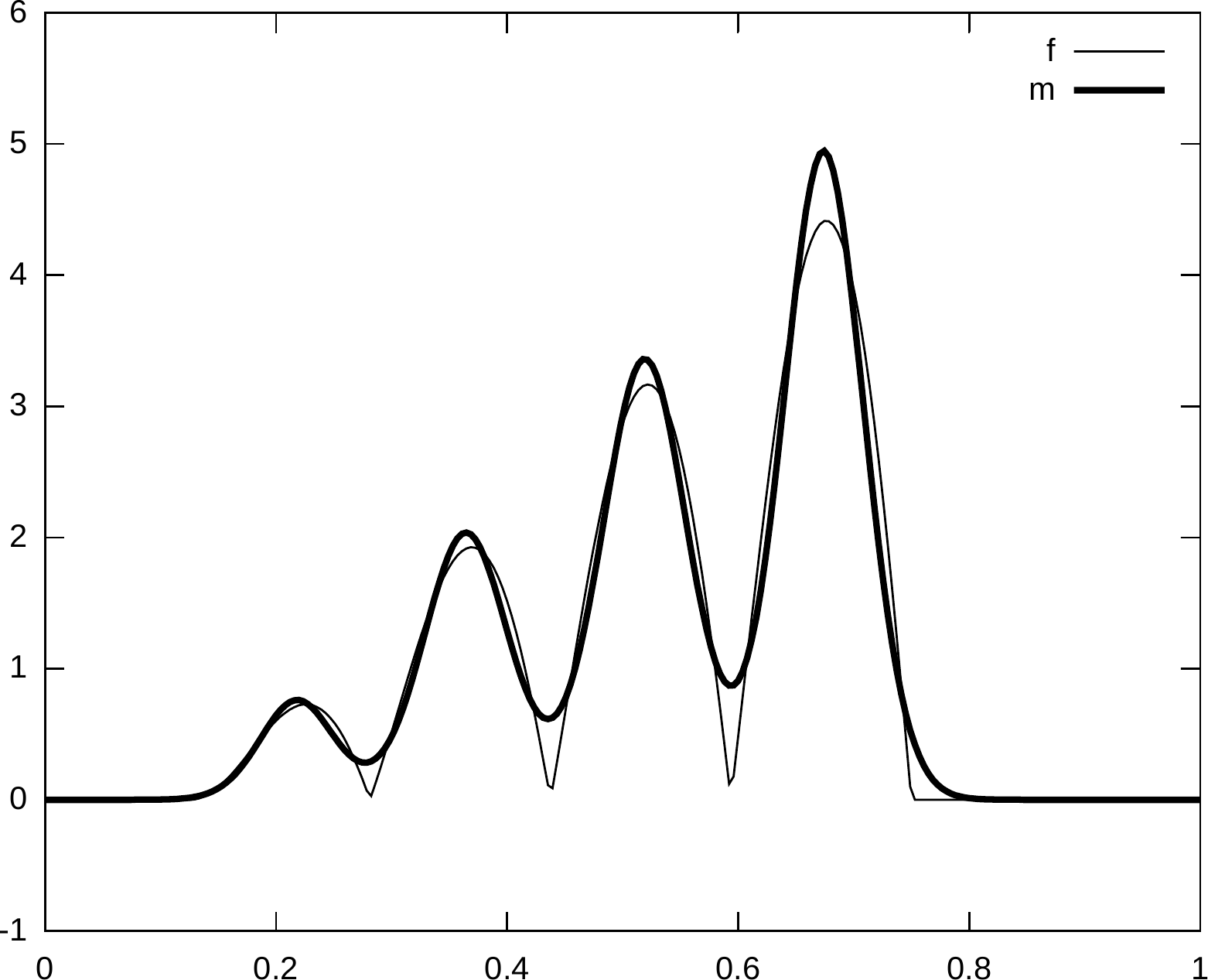}\\
			(c)&(d)
		\end{tabular}
	\end{center}
	\caption{Approximation of an oscillatory data distribution, for $K=1,2,3,4$ mixture components (from (a) to (d)).}\label{Test2}
\end{figure} 

\noindent{\bf Test 3.} 
We present an application of cluster analysis to the so called {\em color quantization} problem in computer graphics. 
Given an image, the aim is to reduce the total number of colors ($2^{24}$ for images in $24$-bit RGB format) to a given number $K$, so that the resulting image 
is as similar as possible to the original one. As pointed out in \cite[Cap. 9]{Bishop}, this problem can be solved by the classical K-means algorithm. Note this approach is quite different from {\em segmentation}, indeed no geometric correlation between pixels is 
considered, but just the color information. Here we apply the MFG clustering algorithm to the case of an image in grey scales. Each pixel in the image contains a level of grey represented by a value in the 
interval $[0,1]$ (typically sampled with $256$ levels, from black $0$ to white $1$), so that the problem is reduced to dimension one, choosing $\Omega=[0,1]$ uniformly 
discretized with $N=256$ nodes. 
To generate the data set distribution $f$, for each $x\in[0,1]$ we count the number of pixels in the image with grey level $x$, then we normalize it to satisfy $\int_0^1 f(x)\,dx=1$. 
In Figure \ref{Test3}, we show the original image and the corresponding distribution $f$, while in Figure \ref{Test3c} we report the results for $K=2,3,5$ clusters respectively. 
We remark that each image is reconstructed from the corresponding mixture by simply using the responsibilities $\{\gamma_k\}_{k=1,...,K}$ to map the single pixel grey value to the barycentre of its most representative cluster.
If the pixel $p$ has grey value $x_p$, then it is mapped to the value $\mu_{k^*}$, \mbox{where 
$k^*=\arg\max_{k=1,...,K}\gamma_k(x_p)$.}\\
\begin{figure}[!h]
	\begin{center}
		\begin{tabular}{cc}
			\includegraphics[width=.475\textwidth]{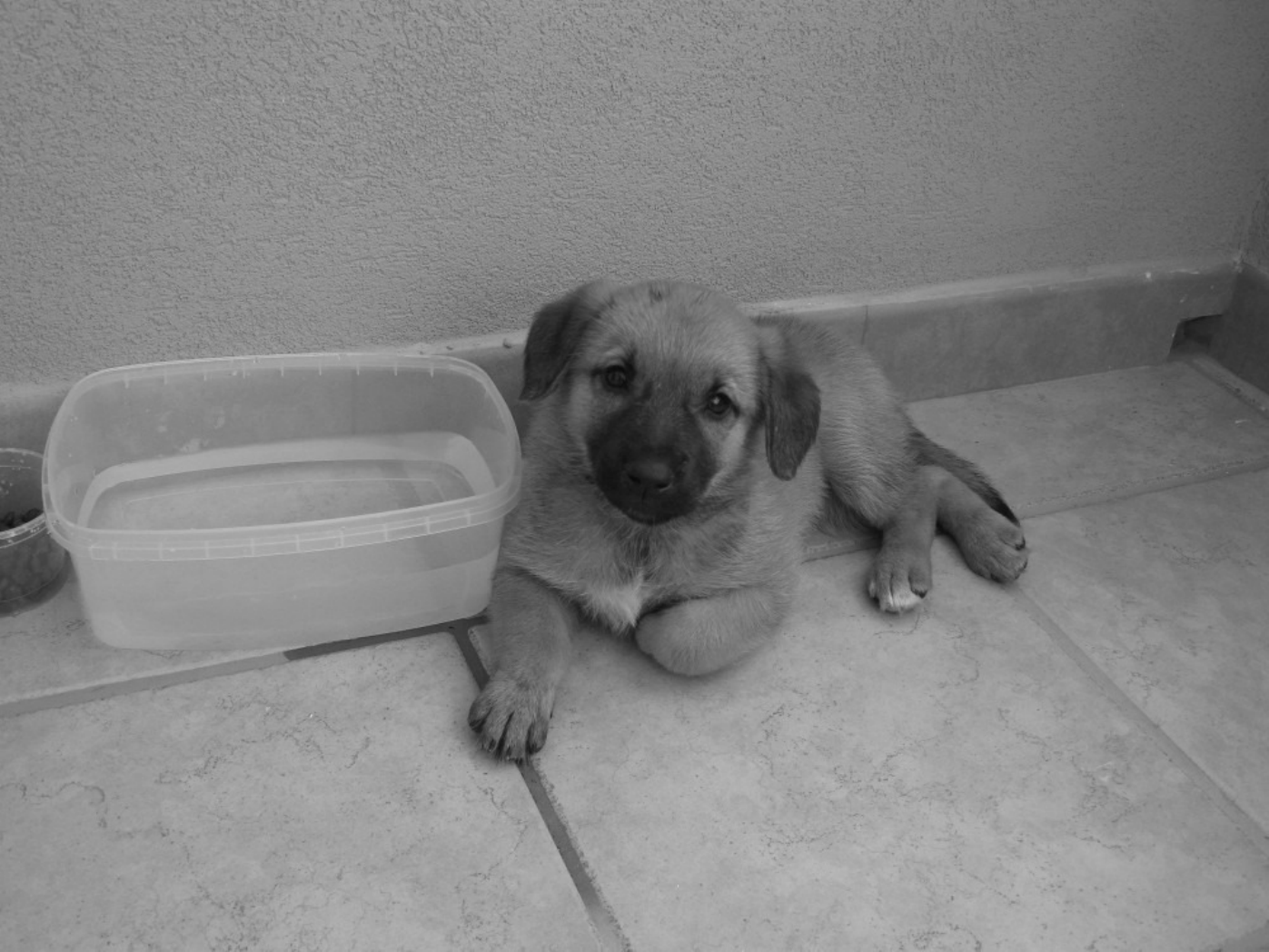}&
			\includegraphics[width=.45\textwidth]{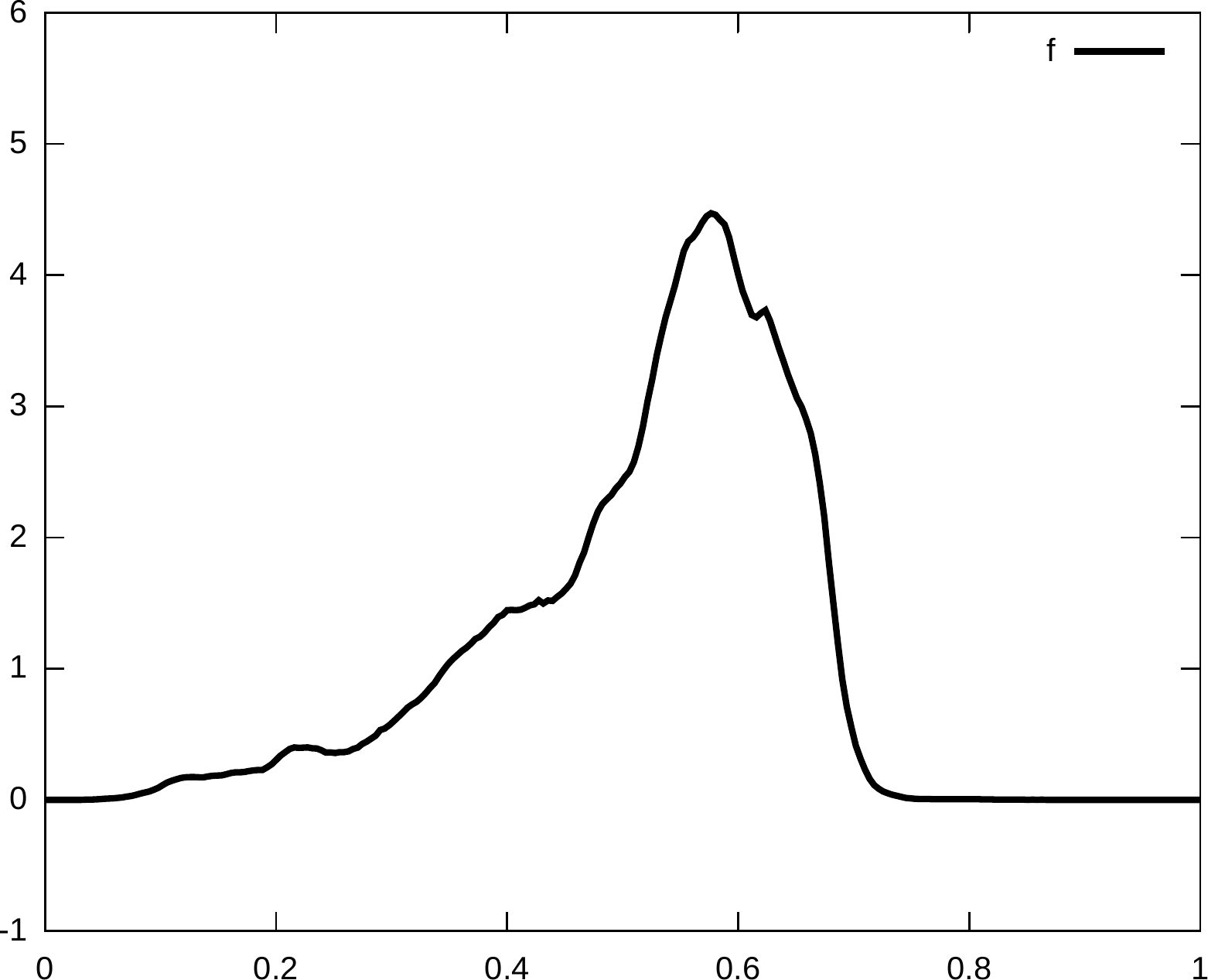}\\
			(a)&(b)
		\end{tabular}
	\end{center}
	\caption{A black and white image (a) and its grey scales distribution (b).}\label{Test3}
\end{figure} 

\begin{figure}[!h]
	\begin{center}
		\begin{tabular}{cc}
			\includegraphics[width=.475\textwidth]{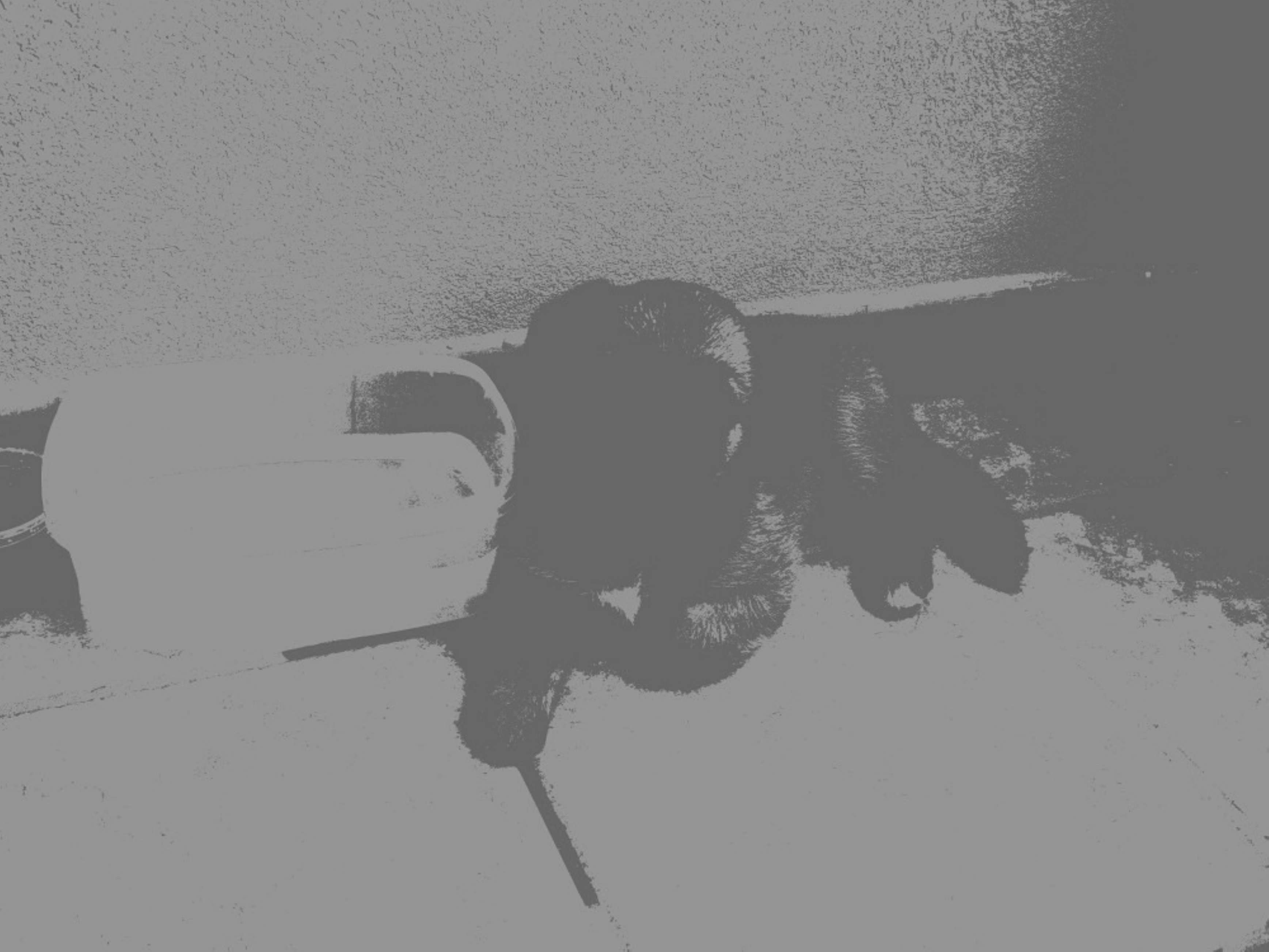}&
			\includegraphics[width=.45\textwidth]{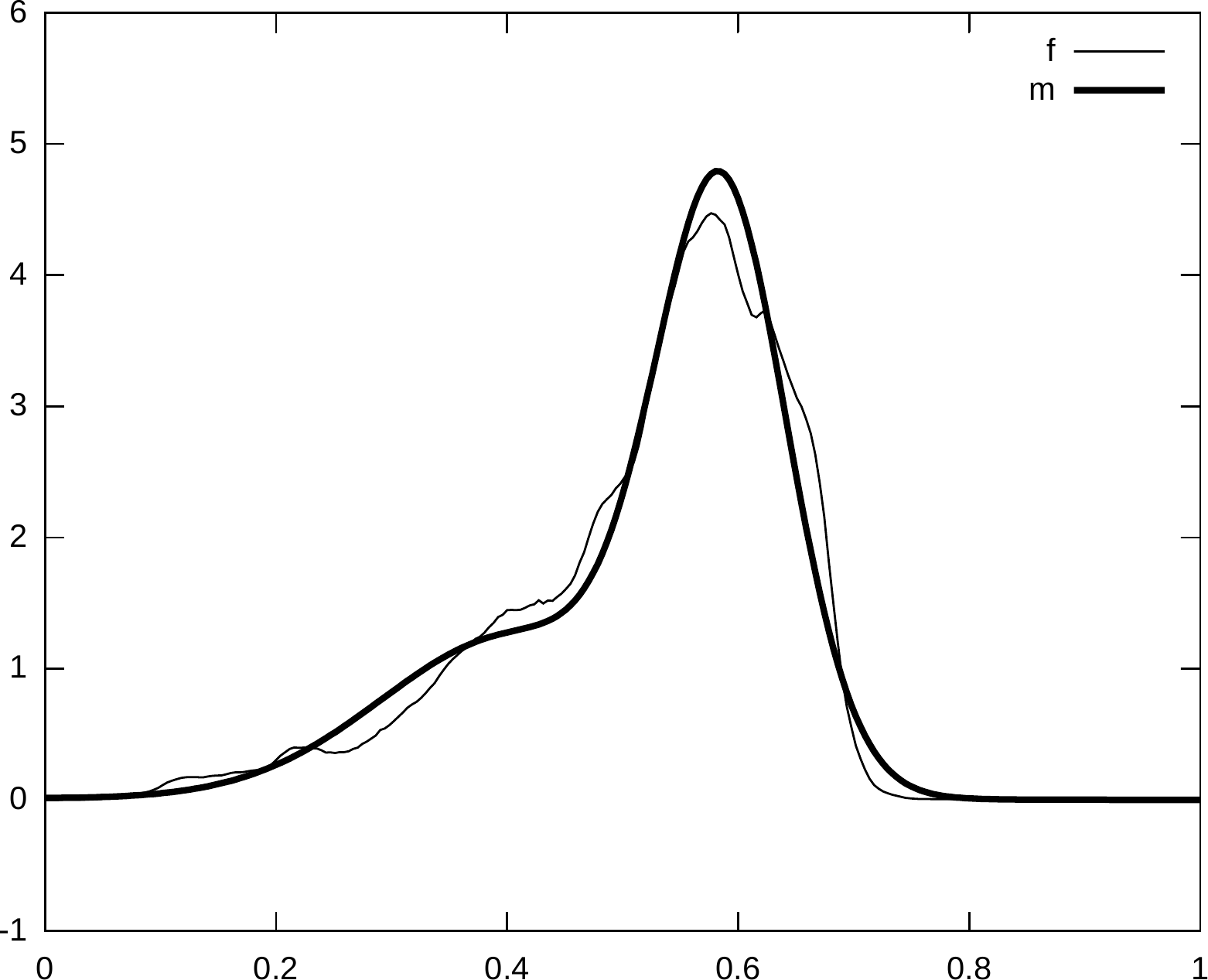}\\
			\includegraphics[width=.475\textwidth]{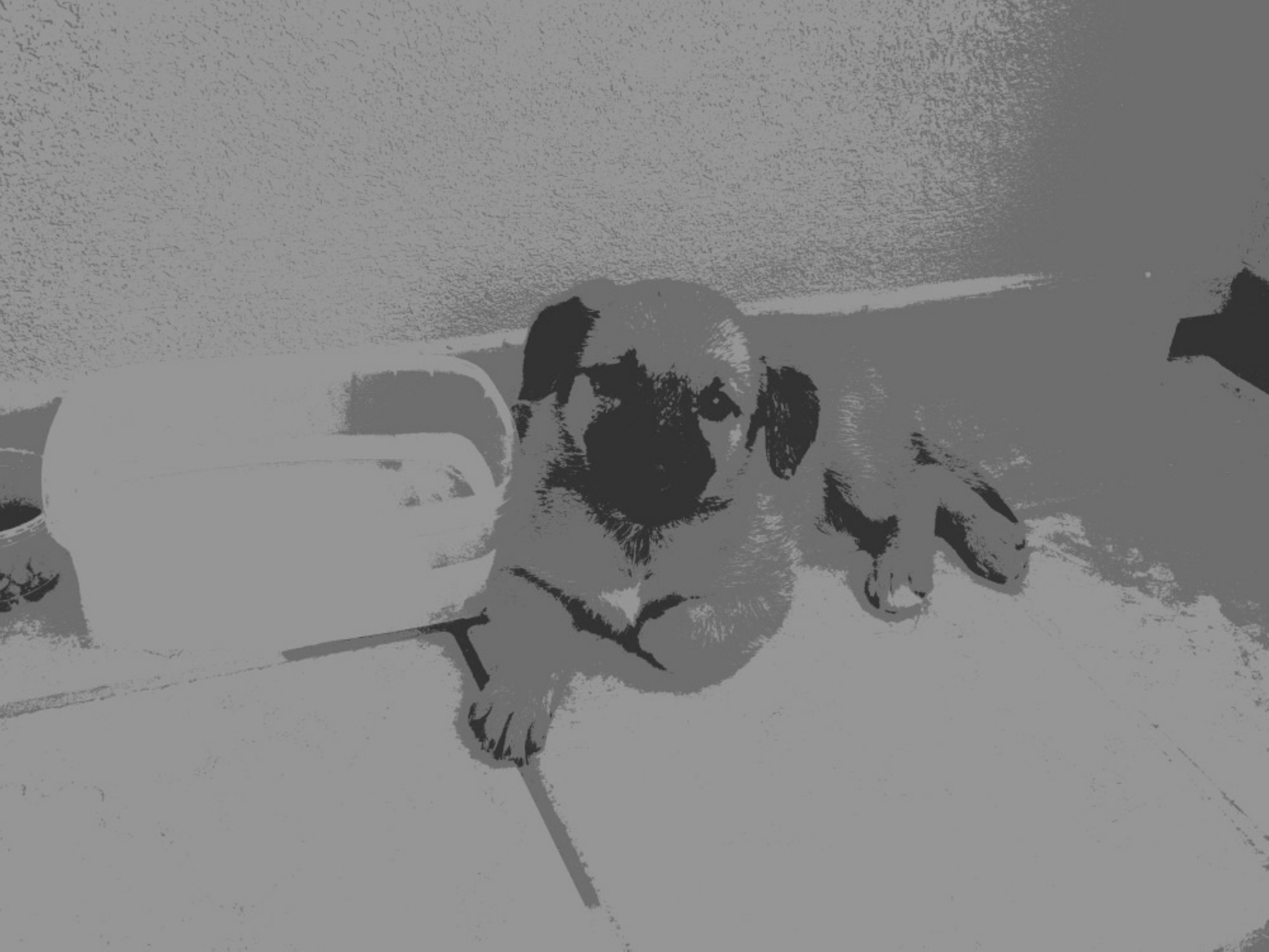}&
			\includegraphics[width=.45\textwidth]{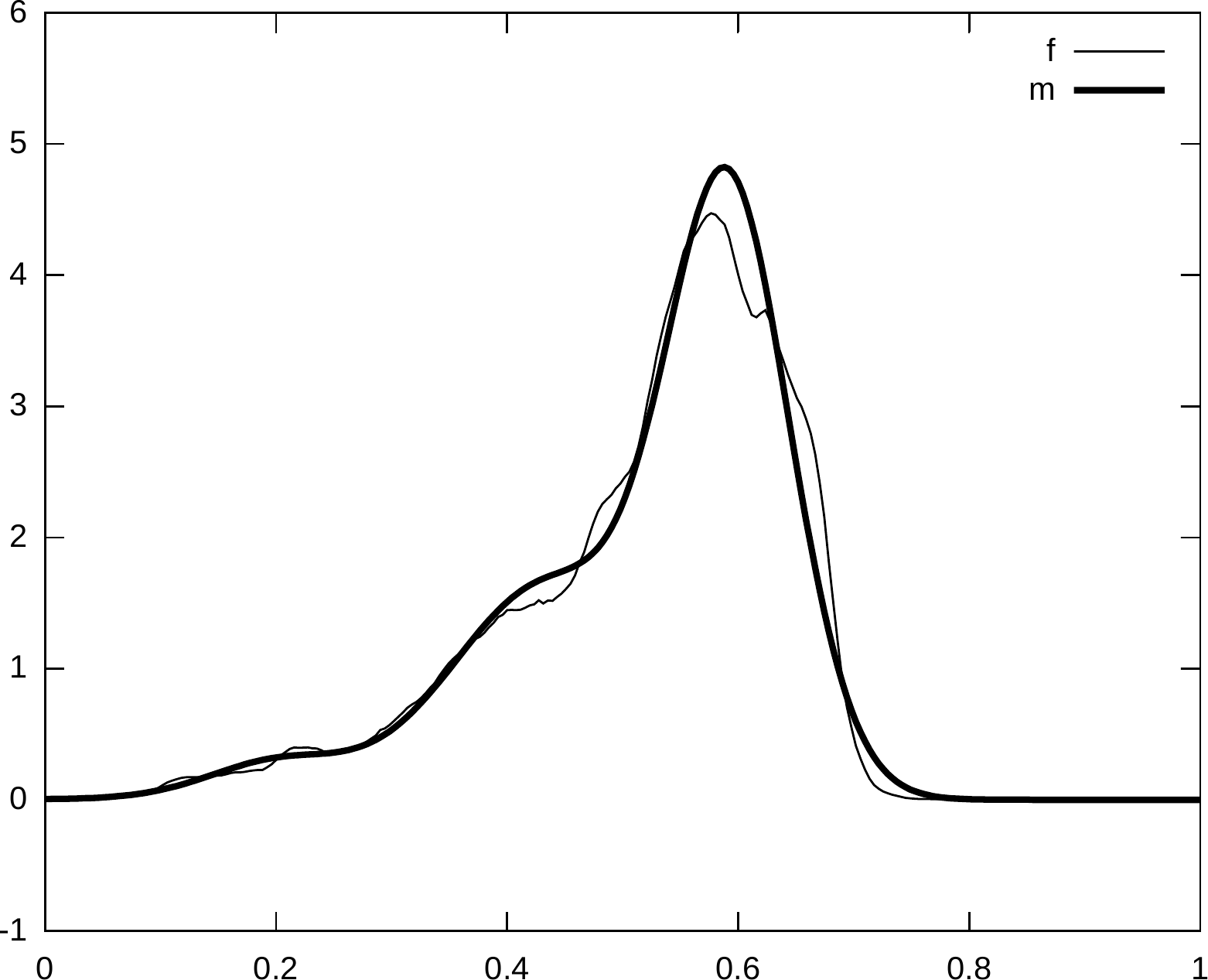}\\
			\includegraphics[width=.475\textwidth]{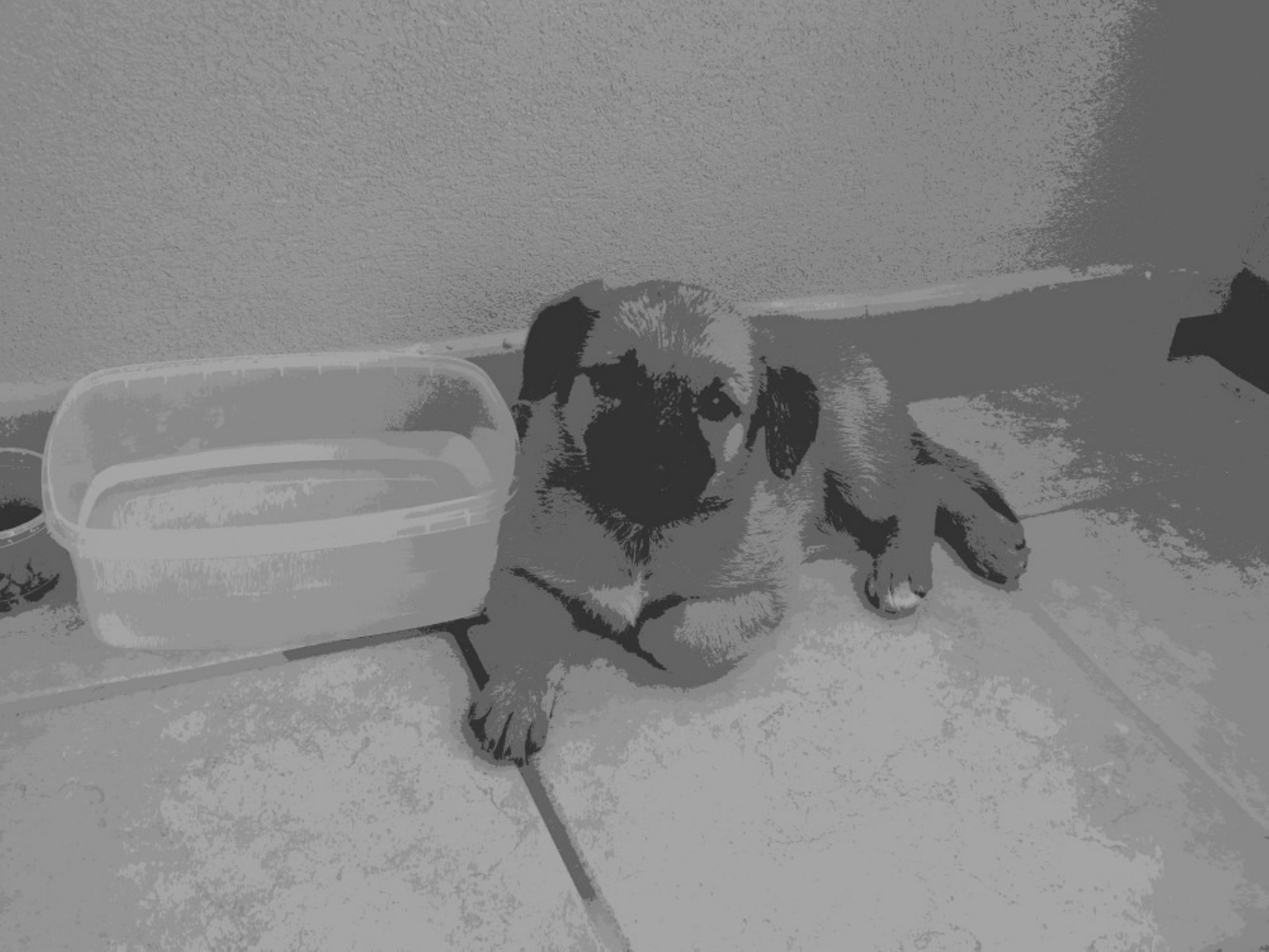}&
			\includegraphics[width=.45\textwidth]{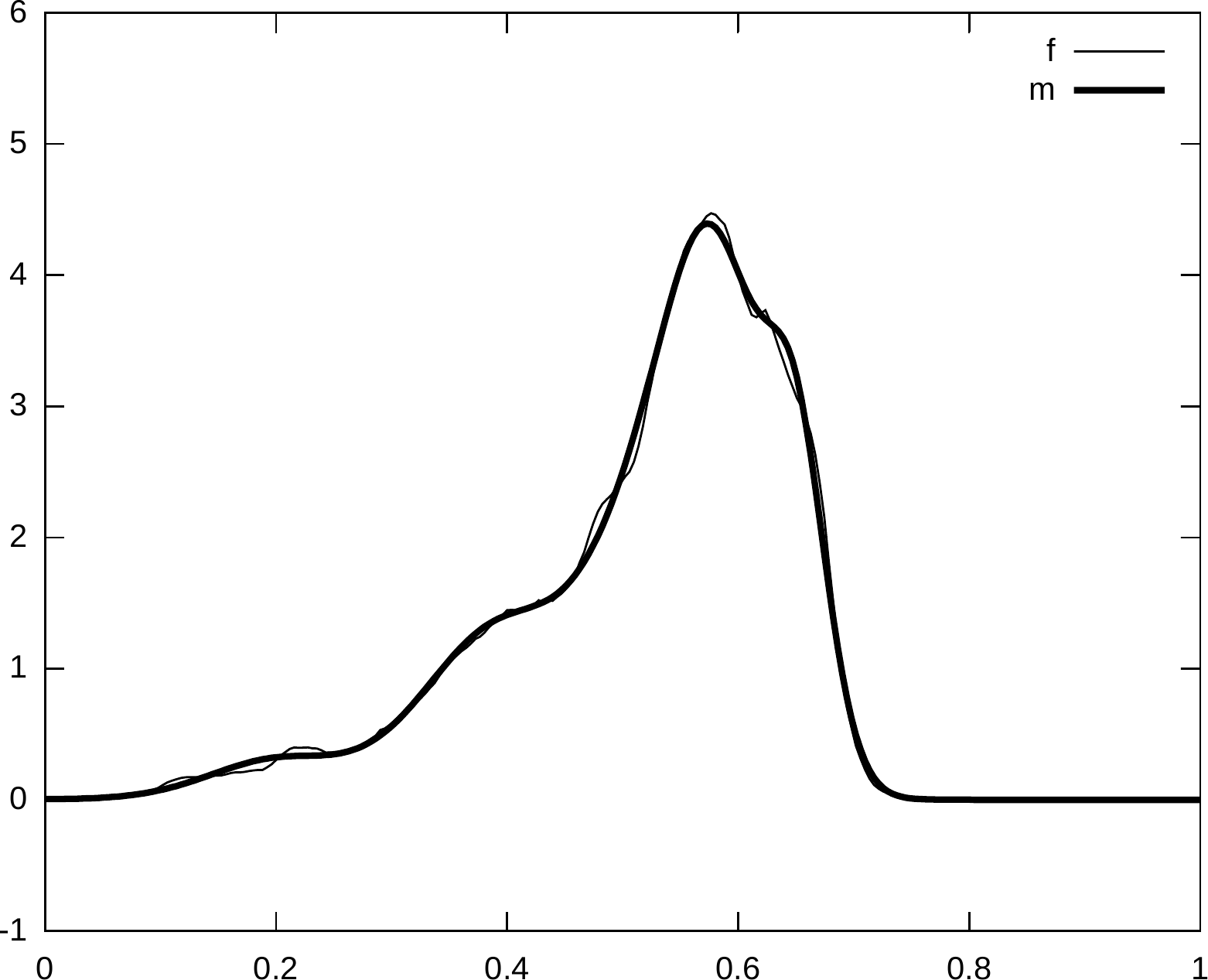}\\
		\end{tabular}
	\end{center}
	\caption{Color quantization via MFG clustering (left panels) and the corresponding mixtures (right panels), for $K=2,3,5$ (from top to bottom panels).}\label{Test3c}
\end{figure} 

\noindent{\bf Test 4. } We finally consider a two dimensional example. We take some data from the Elki project \cite{elki}, 
forming a ``mouse'' similar to a popular comic character. The data set is organized in $3$ clusters (plus some random noise), corresponding 
to the head and the ears of the mouse, see Figure \ref{Test4}. Then, we set $K=3$, we discretize the domain $\Omega=[0,1]^2$ by means of a uniform grid of $N=51^2$ nodes, and we build the 
distribution $f$ as in the previous test, by counting the data points falling in each cell of the grid and normalizing the result to obtain $\int_\Omega f(x)\,dx=1$. 
For the visual representation, we consider RGB triplets in $[0,1]^3$, and we assign to the three clusters the pure colors red $(1,0,0)$, green $(0,1,0)$ and blue $(0,0,1)$ respectively. Then we use the responsibilities $\{\gamma_k\}_{k=1,2,3}\in[0,1]$ to compute the color of each cell of the grid as $C_i=(\gamma_1(x_i),\gamma_2(x_i),\gamma_3(x_i))$, for $i=1,...,N$, which quantifies how much it belongs to a certain cluster. Figure \ref{Test4cluster} shows the surface of the computed mixture and the corresponding clusterization. We observe that the scattered data set is well approximated by the mixture, and  the corresponding three clusters are well separated by  small overlapping regions. 
\begin{figure}[!h]
	\begin{center}
		\begin{tabular}{cc}
			\includegraphics[width=.4\textwidth]{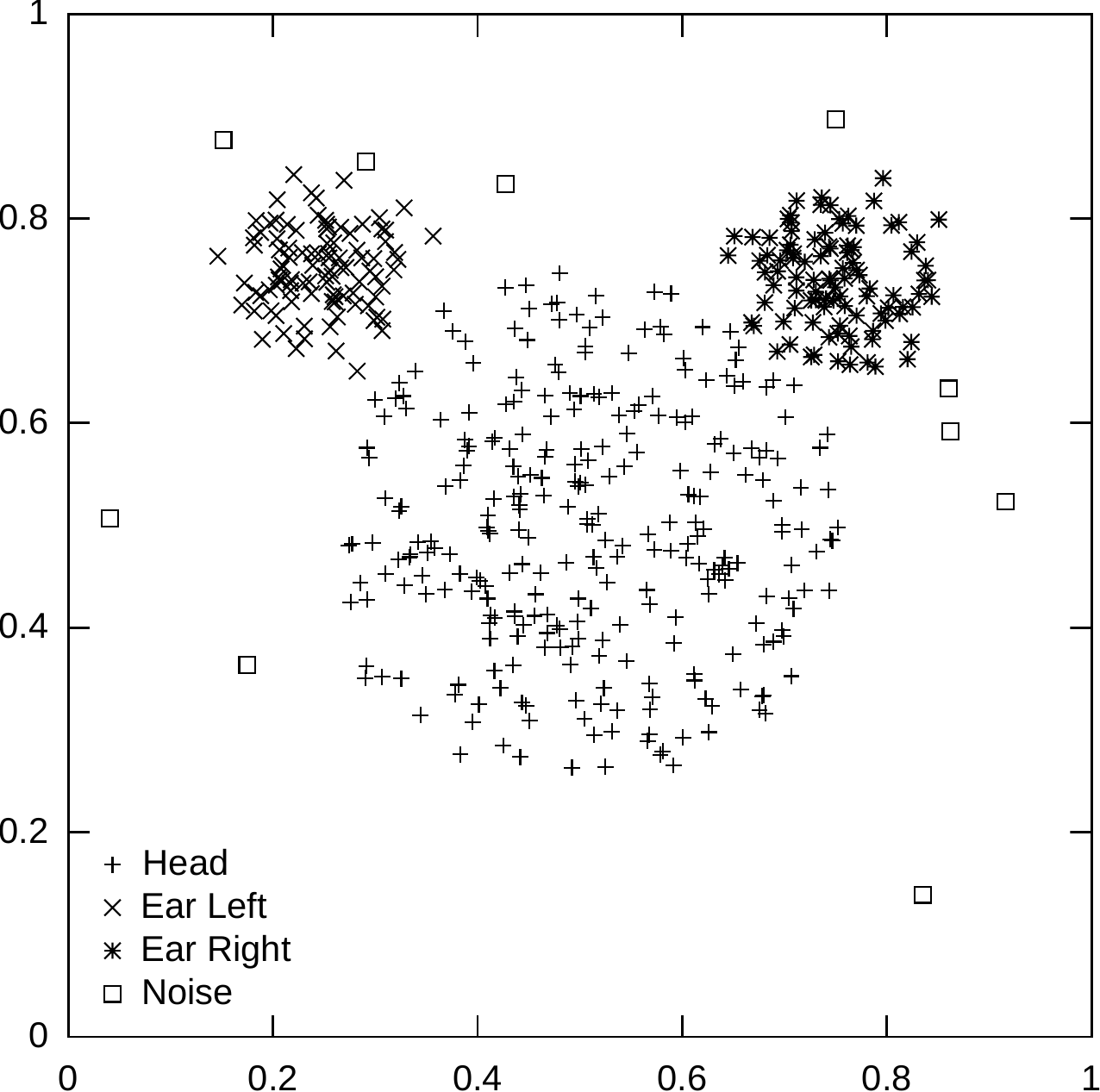}&			\includegraphics[width=.5\textwidth]{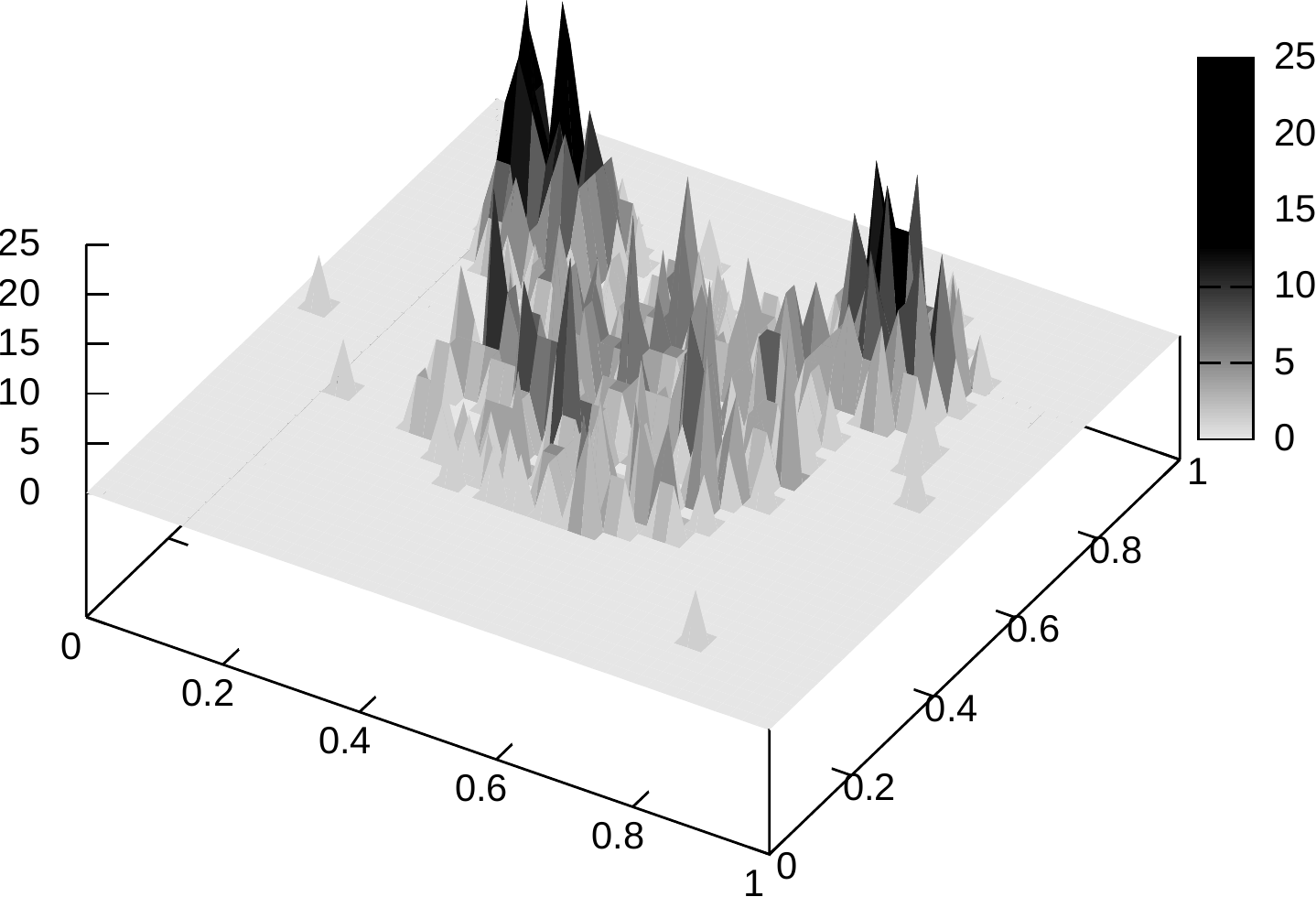}\\
			(a)&(b)
		\end{tabular}
	\end{center}
	\caption{The ``mouse'' data set (a) and the corresponding distribution (b).}\label{Test4}
\end{figure} 

\begin{figure}[!h]
	\begin{center}
		\begin{tabular}{cc}
			\includegraphics[width=.5\textwidth]{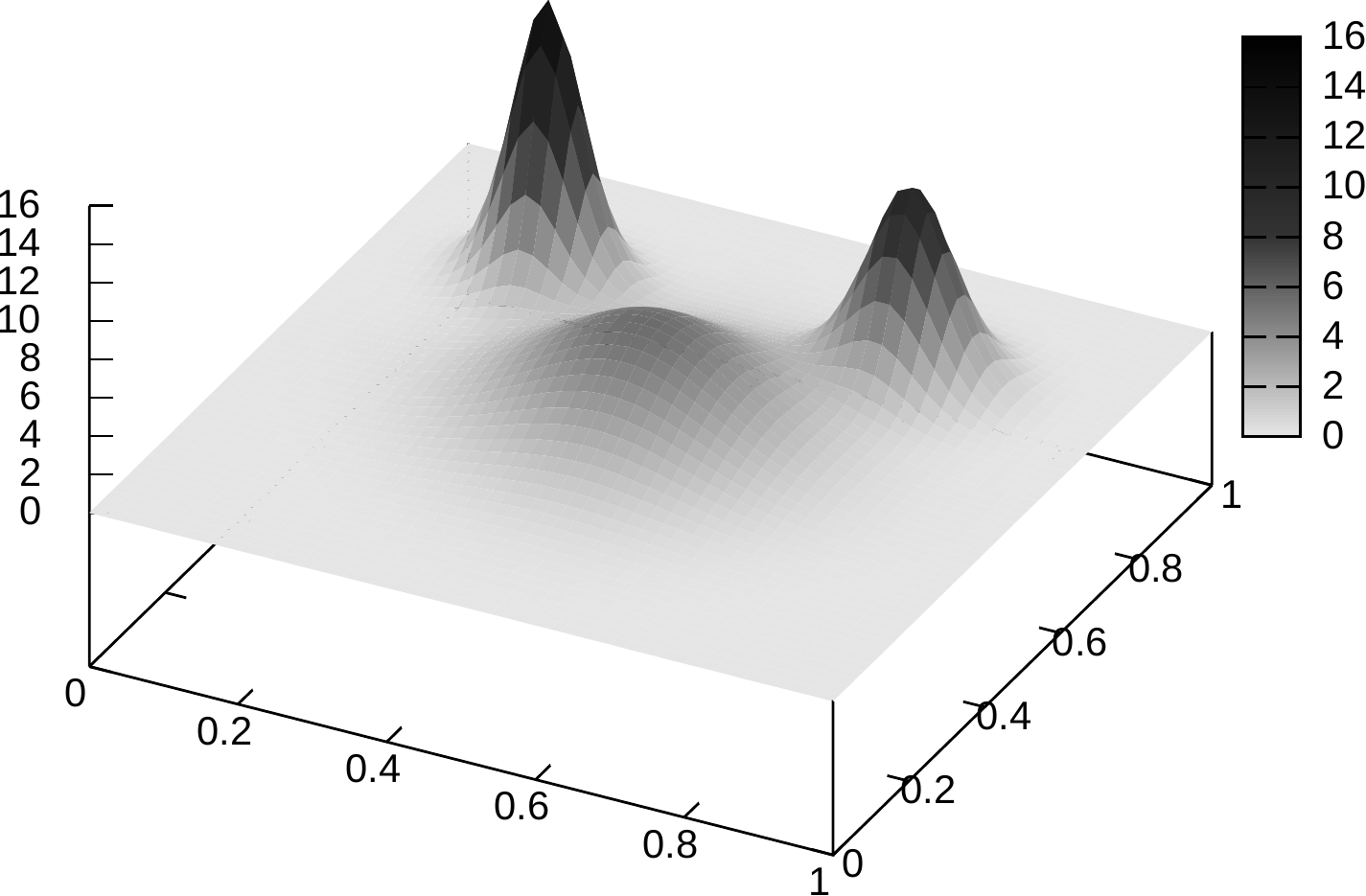}&
			\includegraphics[width=.4\textwidth]{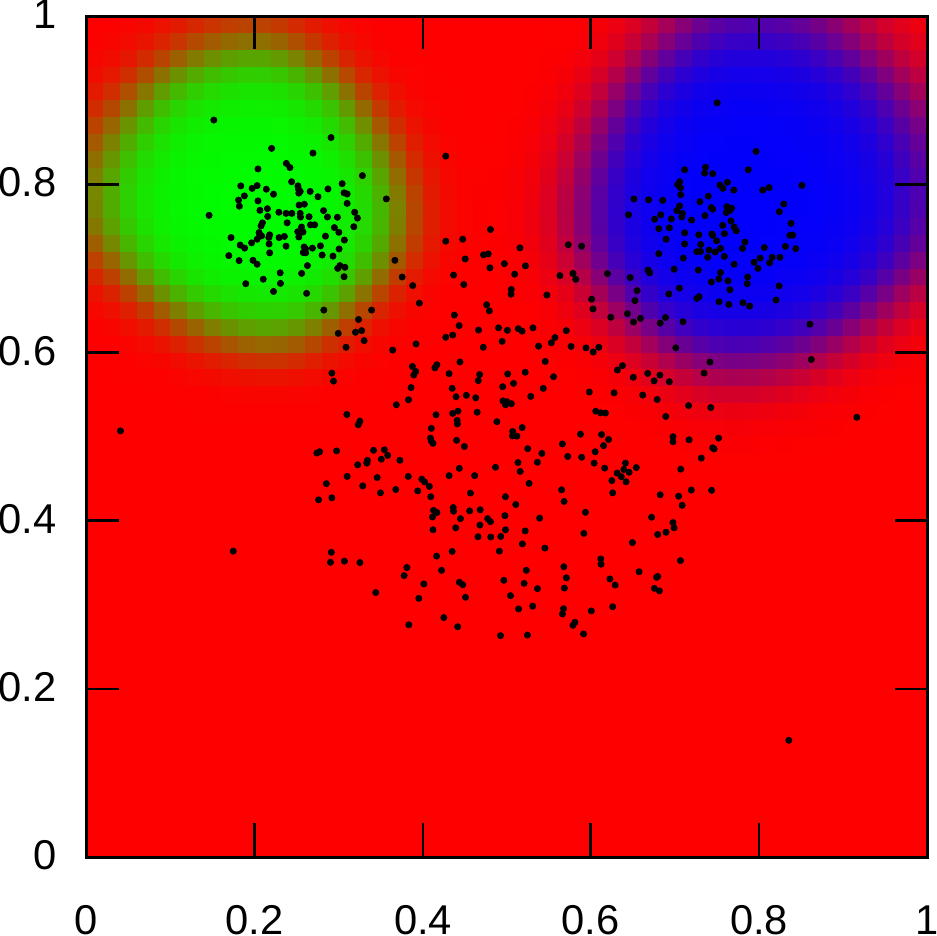}\\
	(a)&(b)
		\end{tabular}
	\end{center}
	\caption{MFG mixture (a) and clustering (b) of the ``mouse'' data set for $K=3$.}\label{Test4cluster}
\end{figure} 


\medskip

\begin{flushright}
	\noindent \verb"laura.aquilanti@sbai.uniroma1.it"\\
	\noindent \verb"fabio.camilli@sbai.uniroma1.it"\\
	SBAI, Sapienza Universit\`{a} di Roma\\
	via A.Scarpa 14, 00161 Roma (Italy)
	
\end{flushright}

\begin{flushright}
	
	\noindent \verb"cacace@mat.uniroma3.it"\\
	Dipartimento di Matematica e Fisica\\
	Universit\`{a} degli Studi Roma Tre\\
	Largo S. L. Murialdo 1, 00146   Roma (Italy)
\end{flushright}


\begin{flushright}
	\noindent \verb"r.demaio@iconsulting.biz"\\
	 IConsulting\\
       Via della Conciliazione 10, 00193  Roma (Italy)
	
\end{flushright}

\end{document}